\documentclass[a4paper,12pt,oneside]{article}
\usepackage{amsfonts}
\usepackage[a4paper]{geometry}
\usepackage{amsmath,amssymb,amsthm,amsfonts}
\usepackage{bbm}
\usepackage{stmaryrd}
\usepackage{amsthm}
\numberwithin{equation}{section}

\usepackage{abstract}
\usepackage{graphicx}
\newtheoremstyle{mystyle}{12pt}{12pt}{\itshape}{0cm}{\bfseries}{}{1em}{}
\theoremstyle{mystyle}
\newtheorem{definition}{Definition}[section]
\newtheorem{theorem}{Theorem}[section]

\newtheorem{lemma}{Lemma}[section]

\newtheorem{corollary}{Corollary}[section]

\newtheorem{remark}{Remark}[section]

\begin{document}

\title{{Asymptotic Joint Distribution of Extreme Sample Eigenvalues and Eigenvectors in the Spiked Population Model}}
 \author{Dai Shi\footnote{Institute of Computational and Mathematical Engineering, Stanford University, Stanford, 94305}}

\date{}

\renewcommand{\thepage}{\roman{page}}
\setcounter{page}{1} 
\maketitle\renewcommand{\thepage}{\arabic{page}}\setcounter{page}{1}

\begin{abstract}
In this paper, we consider a data matrix $X_N\in\mathbb{R}^{N\times p}$ where all the rows are i.i.d. samples in $\mathbb{R}^p$ of mean zero and covariance matrix $\Sigma\in\mathbb{R}^{p\times p}$. Here the population matrix $\Sigma$ is of finite rank perturbation of the identity matrix. This is the ``spiked population model'' first proposed by Johnstone in \cite{21}. As $N, p\to\infty$ but $N/p \to \gamma\in(1, \infty)$, for the sample covariance matrix $S_N := X_NX_N^T/N$, we establish the joint distribution of the largest and the smallest few packs of eigenvalues. Inside each pack, they will  behave the same as the eigenvalues drawn from a Gaussian matrix of the corresponding size. Among different packs, we also calculate the covariance between the Gaussian matrices entries. As a corollary, if all the rows of the data matrix are Gaussian, then these packs will be asymptotically independent. Also, the asymptotic behavior of sample eigenvectors are obtained. Their local fluctuation is also Gaussian with covariance explicitly calculated. 
\end{abstract}

\section{Introduction}\label{sec1}
Suppose we have $N$ independently  and identically distributed samples $x_1, \ldots, x_N\in\mathbb{R}^{p}$. Here $N$ is the sample size and $p$ is the dimension of our data. We can then form the data matrix $X_N = (x_1^T, \ldots, x_N^T)^T \in\mathbb{R}^{N\times p}$ and further define its sample covariance matrix
\[
S_N := \frac{1}{N}X_N^TX_N\in\mathbb{R}^{p\times p}. 
\]
In this paper we are interested in the asymptotic joint distribution of the extreme few eigenvalues and corresponding eigenvectors of the matrix $S_N$. Below is the assumptions of our model.
\begin{itemize}
\item
All the data vectors $x_i$ are independently and identically distributed of mean zero and covariance matrix $\Sigma\in\mathbb{R}^{p\times p}$. Here $\Sigma$ is a non-random positive definite matrix. 
\item
For each  vector $x_i$, the fourth moment $\mathbb{E}|x_{ij}|^4 < \infty$ for $1\leq i\leq N, 1\leq j\leq p$. 
\item
$N, p\to\infty$ but their ratio $N/p := \gamma^2 + o(N^{-1/2})$ where $\gamma$ is a fixed amount in the interval $(1, \infty)$. 
\item
We denote $\ell_1\geq\ell_2\geq\ldots\geq\ell_p$ to be the eigenvalues of the matrix $\Sigma$. Then we assume that all of the $\ell_i$'s are equal to one except for only finite of them. That is, there exist fixed integers $r_+, r_-$ which are independent of $N, p$ such that 
\[
\ell_1 \geq \ell_2\geq\ldots\geq\ell_{r_+} > 1 > \ell_{p-r_-+1} \geq \ell_{p-r_-+2}\geq\ldots\geq\ell_p
\]
and the rest of the eigenvalues are
\[
\ell_{r_++1} = \ell_{r_++2} = \ldots = \ell_{p-r_-} = 1.
\]
\item
We can assume, without losing generality, that our true covariance matrix $\Sigma$ is a diagonal matrix. That is, if we denote $\widehat{\Sigma} := \mathrm{diag}\{\ell_1, \ldots, \ell_{r_+}, \ell_{p-r_-+1}, \ldots, \ell_p\}\in\mathbb{R}^{r_++r_-} = \mathbb{R}^r$ with $r := r_++r_-$, then we can assume $\Sigma$ to be of the form
\begin{equation}
\Sigma := \left(\begin{array}{cc}\widehat{\Sigma} & 0 \\ 0 & I_{p-r}\end{array}\right).
\end{equation}
Hence we can decompose each row of our data matrix $x_i$ to be $x_i^T = (\xi_i^T, \eta_i^T)$, where $\xi_i\in\mathbb{R}^{r}$, $\eta_i\in\mathbb{R}^{p-r}$ and $\mathrm{Cov}(\xi_i) = \widehat{\Sigma}, \mathrm{Cov}(\eta_i) = I_{p-r}$. Our next assumption is that $\eta_i$ has i.i.d. entries and is independent of $\xi$. 
\end{itemize}

The model defined above is the ``spiked population model'' proposed in \cite{21}. The unit eigenvalues represent pure noises, while the spiked eigenvalues represent true information. In real applications, we will encounter such models quite often. In mathematical imaging (see \cite{a1}), the observed spectrum of the sample covariance matrix indeed has some detached eigenvalues, representing the possible scatterers in the region. As another example, in mathematical finance (see \cite{a2}), each row of our data matrix represents the correlated returns of each stock. The sample correlation matrix has some spiked large eigenvalues, representing the main factors driving the market, and some small eigenvalues, representing the linear dependence of these factors. Other possible applications include, but not restricted to, speech recognition (see \cite{a3}), physics mixture (see \cite{a4}) and statistical learning (see \cite{a5}). 

We define $(\widehat{\lambda}^{(N)}_1, \ldots, \widehat{\lambda}^{(N)}_p)$ as the eigenvalues of the sample covariance matrix $S_N$, where $\widehat{\lambda}^{(N)}_1\geq \widehat{\lambda}^{(N)}_2 \geq \ldots\geq\widehat{\lambda}^{(N)}_p$. In the rest of the paper we will refer $\ell_i$ as the true eigenvalues and $\widehat{\lambda}^{(N)}_i$ as the sample eigenvalues. In the null case where $\Sigma = I$, a lot of properties are known.  The empirical measure of $\{\widehat{\lambda}^{(N)}_i\}_{i=1}^p$, denoted by $F_N := \sum_{i=1}^p\delta(\widehat{\lambda}^{(N)}_i)/p$, will almost surely converge in distribution to the Mar\v{c}enko-Pastur law (see \cite{13}), whose density is defined by 
\begin{equation}\label{MP law}
F(\lambda) := \frac{\gamma^2}{2\pi \lambda}\sqrt{(\lambda_+-\lambda)(\lambda-\lambda_-)}\cdot\mathbbm{1}_{\{\lambda_-\leq\lambda\leq\lambda_+\}}
\end{equation}
where $\lambda_+ := (1+\gamma^{-1})^2$ and $\lambda_- := (1-\gamma^{-1})^2$. The support of the density, $[\lambda_-, \lambda_+]$, is often called the Mar\v{c}enko-Pastur sea. Regarding the largest eigenvalue $\widehat{\lambda}^{(N)}_{\max}=\widehat{\lambda}^{(N)}_1$ and the smallest eigenvalue $\widehat{\lambda}^{(N)}_{\min} = \widehat{\lambda}^{(N)}_p$ of the sample covariance matrix $S_N$, German first proved that $\widehat{\lambda}^{(N)}_{\max}\to\lambda_+$ almost surely in \cite{14} and later Silvertein proved that $\widehat{\lambda}^{(N)}_{\min}\to\lambda_-$ almost surely in \cite{15}. That is to say, the largest and the smallest eigenvalues will converge to the corresponding edges of the Mar\v{c}enko-Pastur sea. For a second order approximation, Johansson in  \cite{18} proved that the local fluctuation of $\widehat{\lambda}^{(N)}_{\max}$, properly scaled and centered, will converge weakly to the Tracy-Widom law. Baker, Forrester and Pearce in \cite{20} also proved the similar result for the smallest eigenvalue. We note that most of these results are universal, as is proved in \cite{17}, \cite{23} and \cite{24}, to list a few.

For the spiked population model where $\Sigma\neq I$, the phenomenon becomes much more interesting. Recent research found that the non-null eigenvalues tend to pull the extreme sample eigenvalues out of the Mar\v{c}enko-Pastur sea $[\lambda_-, \lambda_+]$, provided that they are larger or smaller than certain thresholds. 

In \cite{3} Baik and Silverstein proved the almost sure limits of the extreme sample eigenvalues pulled out by the spikes. More precisely, for fixed $j$, almost surely
\begin{eqnarray}
\widehat{\lambda}^{(N)}_j & \to & \left\{
\begin{array}{ll}
\lambda_+ & \text{ if } \ell_j \leq 1+\gamma^{-1}, \\
\ell_j + \gamma^{-2}\ell_j/(\ell_j-1) & \text{ if } \ell_j > 1+\gamma^{-1} .
\end{array}
\right. \label{conv1} \\
\widehat{\lambda}^{(N)}_{p-j} & \to & \left\{
\begin{array}{ll}
\lambda_- & \text{ if } \ell_{p-j} \geq 1-\gamma^{-1}, \\
\ell_{p-j} + \gamma^{-2}\ell_{p-j}/(\ell_{p-j}-1) & \text{ if } \ell_{p-j} < 1-\gamma^{-1} .
\end{array}
\right. \label{conv2}
\end{eqnarray}
Note that this includes the case where some of the $\ell_j$(or $\ell_{p-j}$)'s are the same. In this case, the corresponding $\widehat{\lambda}^{(N)}_j$(or $\widehat{\lambda}^{(N)}_{p-j}$)'s just converge to the same limit specified in (\ref{conv1}) and (\ref{conv2}). We call these eigenvalues ``\emph{packed}''. 

But what is the second order approximation? Baik, Ben Arous and P\'{e}che in \cite{1} observed the phase transition phenomenon of the asymptotic distribution of the largest sample eigenvalue $\widehat{\lambda}^{(N)}_{\max}=\widehat{\lambda}^{(N)}_1$ in the complex Gaussian case. They proved that if $\ell_1 > 1+\gamma^{-1}$ (i.e., when $\widehat{\lambda}^{(N)}_1$ is pulled out of the sea), then the local fluctuation of $\widehat{\lambda}^{(N)}_1$ will be asymptotically the same as the largest eigenvalue of a $k\times k$ GUE matrix, where $k$ is algebraic multiplicity of $\ell_1$. Moreover, in  \cite{4} Bai and Yao obtained the joint local fluctuation of the packed sample eigenvalues --- when suitably centered and scaled, each pack of sample eigenvalues will asymptotically have the same distribution as the eigenvalues of some Gaussian matrix with the corresponding size. 

We note that similar results can also be obtained for perturbed Wigner case, see \cite{2}, \cite{6} and \cite{7} for a reference. 

However, none of these results deal with joint distribution of different packs of sample eigenvalues in our spiked population model. Probably a naive guess is that different packs are asymptotically independent. However, this is not necessarily the case.  By Bai and Yao in \cite{4} the behavior of the local fluctuation of each pack of eigenvalues can be fully described by a Gaussian matrix. Thus, in order to establish the relationship between different eigenvalue packs, we just need to calculate the correlations between entries of these different Gaussian matrices. In this paper, we establish the correlation formula, which is not necessarily zero. Moreover, a sufficient condition for that to be zero is that the first four moments of $x_1$ behave as if they are independent, i.e., for example, $\mathbb{E}(x_{1i}x_{1j})^2 = \mathbb{E} x^2_{1i}x^2_{1j}$. As a corollary, if each row in our data matrix $X$ is Gaussian, then different packs of sample eigenvalues of the sample covariance matrix $S_N = X_N^TX_N/N$ will be asymptotically independent. But this is not in general true for other distributions. We also found that the limiting behavior of the local fluctuation of these eigenvalues only depends on the first moments of $x_i$. We refer to this as the four moment principle. 

Our next interest in this paper is to derive the asymptotic behavior of the eigenvectors. Under the assumption that all the $x_i$'s are Gaussian, Paul in \cite{5} proved that the sample eigenvectors will also be inconsistent --- the angle between the sample eigenvector and the true one will converge to a nonzero constant. In this paper, we remove the Gaussian assumption and establish the local fluctuation of the angle. For this universal case, we also observed the four moment principle --- its asymptotic behavior will also depend only on the first four moments of $x_i$. 

To state our main results below, we introduce the following two notation. 
\begin{definition}
Denote $X_n$ as a random variable. We say $X_n = O_p(1)$ (or bounded in probability) if for any $\epsilon>0$, there exists some $M_\epsilon$ such that 
\[
\mathbb{P}(|X_n|>M_\epsilon) < \epsilon, \qquad \forall n\geq 1. 
\]
We say $X_n = o_p(1)$ (or decay in probability) if for any $\epsilon>0$ we always have 
\[
\lim_{n\to\infty}\mathbb{P}(|X_n|>\epsilon) = 0. 
\]
\end{definition}

The rest of this paper will be organized as follows. In Subsection \ref{sec1.1} and \ref{sec1.2} we will state our main theorems for the limiting behavior of the sample eigenvalues and eigenvectors. Their proofs will be stated in Section \ref{sec2} and Section \ref{sec3}, respectively. All these proofs rely on the central limit theorem for the bilinear form, stated and proved in Section \ref{sec4}. Hence here Section \ref{sec4} only serves as a tool and can be regarded as a self-referenced section. Finally, Section \ref{sec5} serves as an conclusion part of this paper.

\subsection{Main result for eigenvalues}\label{sec1.1}
 
Under the assumptions of the spiked population model, we assume our true covariance matrix $\widehat{\Sigma}$ has the form
\begin{equation}
\widehat{\Sigma} = \mathrm{diag}\{\underbrace{\alpha_1, \ldots, \alpha_1}_{r_1}, \underbrace{\alpha_2, \ldots, \alpha_2}_{r_2}, \ldots, \underbrace{\alpha_q, \ldots, \alpha_q}_{r_q}\}
\end{equation}
with $\sum_{i=1}^qr_i = r$. Also we define $q^*$ such that 
\[
\alpha_1 > \alpha_2 > \ldots \alpha_{q^*} > 1 > \alpha_{q^*+1} > \ldots > \alpha_q. 
\]
Since we are only interested in the isolated eigenvalues lying outside of the Marc\v{e}nko-Pastur sea $[\lambda_-, \lambda_+]$, we can assume without losing generality that none of these $\alpha_i$'s are in the interval $[1-\gamma^{-1}, 1+\gamma^{-1}]$. Then by \cite{3} there will be $r$ isolated sample eigenvalues lying outside of the Marc\v{e}nko-Pastur sea.  Recall that the the sample eigenvalues are denoted by $\{\widehat{\lambda}^{(N)}_i\}_{i=1}^r$. Since we wish to denote both the greatest and the smallest few eigenvalues in a convenient way, with slight abuse of notations we define $\{\widehat{\lambda}^{(j)}\}_{j=1}^r$ by
\[
\widehat{\lambda}^{(j)} = \left\{
\begin{array}{ll}
\widehat{\lambda}^{(N)}_j & \text{ if } 1\leq j\leq \sum_{i=1}^{{q^*}}r_i, \\
\widehat{\lambda}^{(N)}_{p-r+j} & \text{ if } \sum_{i=1}^{{q^*}}r_i+1\leq j\leq r.
\end{array} 
\right.
\]
Then we know that these $\widehat{\lambda}^{(j)}$'s can be grouped into $q$ packs with
\[
\widehat{\lambda}^{(\sum_{i=1}^{j-1}r_i + s)} \to \rho_{\alpha_j} := \alpha_j + \frac{\gamma^{-2}\alpha_j}{\alpha_j-1}, \qquad\text{ for } 1\leq s\leq r_j. 
\]
To consider the second order approximation, for any $1\leq j\leq q$ define $Z^{(N)}_j \in\mathbb{R}^{r_j}$ by
\begin{eqnarray}
Z^{(N)}_j & = & \Big(Z^{(N)}_{j, 1}, \ldots, Z^{(N)}_{j, r_j}\Big) \\
Z^{(N)}_{j, s} & = & \sqrt{N}\Big(\widehat{\lambda}^{(\sum_{i=1}^{j-1}r_i + s)} - \rho_{\alpha_j}\Big), \qquad 1\leq s\leq r_j. 
\end{eqnarray}

Here's the asymptotic behavior of the $r$-dimensional vector $Z^{(N)} := (Z^{(N)}_1, \ldots, Z^{(N)}_q)$.

%
%
\begin{theorem}\label{main theorem real}
As $N\to \infty$, $Z^{(N)}$ will weakly converge to the $r$ dimensional vector $Z := (Z_1, \ldots, Z_q)$ partitioned in the same way as $Z^{(N)}$, such that
\begin{itemize}
\item
For each $1\leq j\leq q$, $Z_j \in\mathbb{R}^{r_j}$ has the same distribution as the $r_j$ eigenvalues of a $r_j\times r_j$ Gaussian symmetric matrix $(1+\gamma^{-2}\alpha_jm_3(\rho_{\alpha_j}))^{-1}\cdot G^{(j)}$ with each entry $G^{(j)}_{st}$ being Gaussian distributed with mean zero. 
\item
The intra-matrix-covariance of these Gaussian entries are 
\begin{eqnarray*}
\mathrm{Cov}(G^{(j)}_{st}, G^{(j)}_{uv}) & = & 
\widetilde{\omega}_{jj}\Big[\mathbb{E}[\xi_{\sum_{i=1}^{j-1}r_i + s, 1}\xi_{\sum_{i=1}^{j-1}r_i+u, 1}\xi_{\sum_{i=1}^{j-1}r_i + t, 1}\xi_{\sum_{i=1}^{j-1}r_i + v, 1}] - \alpha_j^2\mathbbm{1}_{s=t, u = v}\Big] \\
&&  + (\widetilde{\theta}_{jj}-\widetilde{\omega}_{jj})\Big[\alpha_j^2\mathbbm{1}_{s=v, u=t} + \alpha_j^2\mathbbm{1}_{s=u, t=v}\Big]
\end{eqnarray*}
for any $1\leq s\leq t\leq r_j, 1\leq u\leq v \leq r_j$. 
\item
The inter-matrix-covariances of these Gaussian entries are
\begin{eqnarray*}
\mathrm{Cov}(G^{(j)}_{st}, G^{(j')}_{uv}) & = & 
\widetilde{\omega}_{jj'}\Big[\mathbb{E}[\xi_{\sum_{i=1}^{j-1}r_i + s, 1}\xi_{\sum_{i=1}^{j'-1}r_i+u, 1}\xi_{\sum_{i=1}^{j-1}r_i + t, 1}\xi_{\sum_{i=1}^{j'-1}r_i + v, 1}] - \alpha_j\alpha_{j'}\mathbbm{1}_{s=t, u = v}\Big]
\end{eqnarray*}
for any $1\leq s \leq t\leq r_j$, $1\leq u\leq v\leq r_{j'}$ and for any $1\leq j\neq j'\leq q$. 
\end{itemize}
Here $\widetilde{\omega}_{jj'}, \widetilde{\theta}_{jj'}$ are defined by
\begin{eqnarray*}
\widetilde{\omega}_{jj'} & = & \biggl(1+\frac{\gamma^{-2}}{\alpha_j-1}\biggl)\biggl(1+\frac{\gamma^{-2}}{\alpha_{j'}-1}\biggl).\\
\widetilde{\theta}_{jj'} & = & \frac{(\alpha_{j}-1+\gamma^{-2})(\alpha_{j'}-1+\gamma^{-2})}{(\alpha_j-1)(\alpha_{j'}-1)-\gamma^{-2}}.
\end{eqnarray*}
and $m_3(\rho_{\alpha_j})$ is defined in (\ref{m3})
\end{theorem}

\begin{remark}
In particular, if for any $j\neq j'$ and any $1\leq s, t\leq r_j$, $1\leq u, v\leq r_{j'}$, 
\begin{equation}\label{cond}
\mathbb{E}[\xi_{\sum_{i=1}^{j-1}r_i + s, 1}\xi_{\sum_{i=1}^{j'-1}r_i+u, 1}\xi_{\sum_{i=1}^{j-1}r_i + t, 1}\xi_{\sum_{i=1}^{j'-1}r_i + v, 1}]  = \alpha_j\alpha_{j'}\mathbbm{1}_{s=t, u = v},
\end{equation}
then these $q$ Gaussian matrices are independent of each other. In this case, the $q$ packs of eigenvalues are asymptotically independent of each other.  If all the $\xi_i$'s are i.i.d. Gaussian, then (\ref{cond}) holds true, as for Gaussian distribution uncorrelated-ness implies independence. However, for general distribution of $\xi_i$'s, (\ref{cond}) no longer holds true. Hence in general these $q$ packs are not independent of each other. 
\end{remark}

\begin{remark}
We observe the four moment principle in Theorem \ref{main theorem real}. The local fluctuation for the sample eigenvalues $\lambda^{(j)}$ only depend on the first four moments of $\xi_1$. The similar phenomena has been observed for many times in random matrix literature. 
\end{remark}

\subsection{Main result for eigenvectors}\label{sec1.2}
In the subsection we only consider the case where all the eigenvalues of $\widehat{\Sigma}$ are distinct. That is, all the eigenvalues of $\mathrm{Cov}(\xi)$ have multiplicity one. In this case $q=r$ and 
\[
\widehat{\Sigma} = \mathrm{diag}\{\alpha_1, \ldots, \alpha_r\}. 
\]
Again, since we are only interested in the isolated packs of eigenvalues, we can assume without losing generality that all the $\alpha_i$'s are sufficiently far away from 1, i.e. $|\alpha_i-1|>\gamma^{-1}$. 

Recall we can decompose our sample point $x_i$ as $x_i^T = (\xi_i^T, \eta_i^T)$. Here we use the same partition for the sample eigenvectors. Indeed, we define the $j$-th eigenvector by $(\widehat{u}^{(j)T}, \widehat{v}^{(j)T})^T$. That is,
\[
S_N\left(\begin{array}{c}\widehat{u}^{(j)} \\ \widehat{v}^{(j)}\end{array}\right) = \widehat{\lambda}^{(j)}\left(\begin{array}{c}\widehat{u}^{(j)} \\ \widehat{v}^{(j)}\end{array}\right), \qquad \widehat{u}^{(j)}\in\mathbb{R}^{r},\quad \widehat{v}^{(j)}\in\mathbb{R}^{p-r}
\]
where $S_N$ is the sample covariance matrix. Since the eigenvectors are unique up to scaling, we require that $\|\widehat{u}^{(j)}\|_2 = 1$ and $\widehat{u}^{(j)}_j \geq 0$. Here $\widehat{u}^{(j)}_j$ is the $j$-th entry of $\widehat{u}^{(j)}$. Also, for notational convenience, we denote $\widehat{u}^{(j)}_{-j}$ as the $(r-1)$-dimensional vector obtained by deleting the $j$-th entry of $\widehat{u}^{(j)}$.

Note that the $j$-th true eigenvector is $(e_j^T, 0^T)^T$, where $e_j = (0, \ldots, 0, 1, 0, \ldots, 0)\in\mathbb{R}^r$ is a vector of all zeros except the $j$-th entry being one. Hence intuitively we should have $\widehat{u}^{(j)}\to e_j$. This is indeed the case from the following theorem.

%
%
\begin{theorem}\label{thm2}
We have $u^{(j)}\to e_j$ in probability. Furthermore, $u^{(j)}_j = 1+O_p(N^{-1})$. For the second order approximation of the $u^{(j)}_{-j}$'s we have the following result. 

Jointly, 
\begin{eqnarray*}
\sqrt{N} \cdot u^{(j)}_i & \to & \frac{\alpha_i}{\rho_{\alpha_j}(\alpha_j-\alpha_i)}G^{(j)}_i, \qquad, i\neq j\\
\sqrt{N}\cdot(\lambda^{(j)}-\rho_{\alpha_j}) & \to & \frac{1}{1+\gamma^{-2}\alpha_jm_3(\rho_{\alpha_j})}G^{(j)}_j.
\end{eqnarray*}
Here the entries $\{G^{(j)}_i\}_{i, j}$ are Gaussian distributed with mean zero and covariance
\[
\mathrm{Cov}(G^{(j)}_i, G^{(j')}_{i'}) = \widetilde{\omega}_{jj'}\Big[\mathbb{E}[\xi_i\xi_{i'}\xi_{j}\xi_{j'}]-\alpha_{i}\alpha_{i'}\mathbbm{1}_{i=j, i'=j'}\Big] + (\widetilde{\theta}_{jj'}-\widetilde{\omega}_{jj'})[\alpha_{i}\alpha_j\mathbbm{1}_{i=j', j=i'}+\alpha_{i}\alpha_j\mathbbm{1}_{i=i', j = j'}]
\]
where $\widetilde{\omega}_{jj'}, \widetilde{\theta}_{jj'}$ and $m_3(\rho_{\alpha_j})$ are defined in Theorem \ref{main theorem real}. 
\end{theorem}

\begin{remark}
If in particular $\mathbb{E}[\xi_i\xi^3_{j}] =  0$ for some $i\neq j$, then the local fluctuation of $u^{(j)}_{-j}$ and $\lambda^{(j)}$ will be asymptotically independent. This condition is satisfied if we assume that the distribution of the sample points $x_i$ are i.i.d. Gaussian. But this is not necessarily true for general distributions. 
\end{remark}

From Theorem \ref{thm2}, we observe that the $\widehat{u}$ part will be asymptotically consistent. But what about the whole eigenvector $(\widehat{u}^T, \widehat{v}^T)^T?$ The next theorem shows that the whole eigenvector will not be consistent --- the angle between the sample eigenvector and the true eigenvector will converge to a non-vanishing constant.

%
%
\begin{theorem}\label{thm3}
Denote $\widehat{\theta}^{(j)}$ as the angle between the $j$-th sample eigenvector $(\widehat{u}^{(j)^T}, \widehat{v}^{(j)^T})^T$ and the $j$-th true eigenvector $(e_j^T, 0^T)^T$. Then for all $1\leq j\leq r$, we have that jointly 
\begin{multline*}
\Big(1+\gamma^{-2}m_3(\rho_{\alpha_j})\alpha_j\Big)^{3/2}\cdot\sqrt{N}\biggl(\cos\widehat{\theta}^{(j)} - \frac{1}{\sqrt{1+\gamma^{-2}m_3(\rho_{\alpha_j})\alpha_j}}\biggl) \\
\stackrel{D}{\to}\frac{\gamma^{-2}m_4(\rho_{\alpha_j})\alpha_j}{1+\gamma^{-2}m_3(\rho_{\alpha_j})\alpha_j}G_j + H_j
\end{multline*}
where $\{G_{j}, H_{j'}\}_{j=1}^r$ are jointly normal with mean zero and covariance
\begin{eqnarray*}
\mathrm{Cov}(G_j, G_{j'}) & = & \widetilde{\omega}_{jj'}[\mathbb{E}\xi_j^2\xi_{j'}^2 - \alpha_{j}\alpha_{j'}] + 2(\widetilde{\theta}_{jj'}-\widetilde{\omega}_{jj'})\alpha_j^2\mathbbm{1}_{j=j'}, \\
\mathrm{Cov}(H_j, H_{j'}) & = & \widetilde{\zeta}_{jj'}[\mathbb{E}\xi_j^2\xi_{j'}^2 - \alpha_{j}\alpha_{j'}] + 2(\widetilde{\tau}_{jj'}-\widetilde{\zeta}_{jj'})\alpha_j^2\mathbbm{1}_{j=j'}, \\
\mathrm{Cov}(G_{j}, H_{j'}) & = & \widetilde{\kappa}_{jj'}[\mathbb{E}\xi_j^2\xi_{j'}^2 - \alpha_{j}\alpha_{j'}] + 2(\widetilde{\mu}_{jj'}-\widetilde{\kappa}_{jj'})\alpha_j^2\mathbbm{1}_{j=j'}. \\
\end{eqnarray*}
Here $\widetilde{\omega}_{jj'}, \widetilde{\theta}_{jj'}. \widetilde{\zeta}_{jj'}, \widetilde{\tau}_{jj'}, \widetilde{\kappa}_{jj'}$ and $\widetilde{\mu}_{jj'}$ are defined in (\ref{real1}), (\ref{real2}), (\ref{zzeta}), (\ref{ttau}), (\ref{kkappa}) and (\ref{mmu}), respectively. The functions $m_3(\rho_{\alpha_j}), m_4(\rho_{\alpha_j})$ are defined in (\ref{m3}) and (\ref{m4}).
\end{theorem}

\begin{remark}
Here $G_j$ is exactly the same as $G^{(j)}_j$ in Theorem \ref{thm2}. This can build the correlation between the local fluctuation of the sample eigenvectors and the angles. 
\end{remark}

\begin{remark}
Once again, in Theorem \ref{thm2} and Theorem \ref{thm3} we observed the four moment principle. The local fluctuation of the sample eigenvectors as well as the angles only depends on the first four moments of $\xi_i$. 
\end{remark}

\section{Asymptotic behavior for eigenvalues}\label{sec2}

Following the notation in Section \ref{sec1}, each sample point can be decomposed as $x_i^T = (\xi_i^T, \eta_i^T)$. Our data matrix $X_N$ has the form
\begin{equation}\label{eqn2.1}
X_N = \left(\begin{array}{c}x_1^T \\ \vdots \\ x_N^T\end{array}\right) = \left(\begin{array}{cc}\xi_1^T & \eta_1^T \\ \vdots & \vdots \\ \xi_N^T & \eta^T_N \end{array}\right),
\end{equation}
where $\xi_i, \eta_i$ are independent, $\mathrm{Cov}(\xi_i) = \widehat{\Sigma}$ and $\mathrm{Cov}(\eta_i) = I$. We introduce the notation 
\begin{equation}\label{eqn2.2}
X_\xi = \frac{1}{\sqrt{N}}\left(\begin{array}{c}\xi_1^T \\ \vdots \\ \xi_N^T\end{array}\right), \qquad  X_\eta = \frac{1}{\sqrt{N}}\left(\begin{array}{c}\eta_1^T \\ \vdots \\ \eta_N^T\end{array}\right)
\end{equation}
Then $X_N = \sqrt{N}(X_\xi, X_\eta)$. Our sample covariance matrix is then
\begin{equation}\label{eqn2.3}
S_N = \frac{1}{N}X_N^TX_N = \left(\begin{array}{cc}X_\xi^TX_\xi & X_\xi^TX_\eta \\ X_\eta^TX_\xi & X_\eta^TX_\eta\end{array}\right).
\end{equation}
To calculate its eigenvalue, we calculate 
\begin{eqnarray}
\det(\lambda I - S_N) & = & \det\left(\begin{array}{cc}\lambda I- X_\xi^TX_\xi & -X_\xi^TX_\eta \\ -X_\eta^TX_\xi & \lambda I-X_\eta^TX_\eta \end{array}\right) \nonumber \\
& = & \det\left(\begin{array}{cc}\lambda I- X_\xi^TA_N(\lambda)X_\xi & 0 \\ 0 & \lambda I-X_\eta^TX_\eta \end{array}\right)  \label{eqn2.4}
\end{eqnarray}
where
\begin{equation}\label{eqn2.5}
A_N(\lambda) = I + X_\eta(\lambda I-X_\eta^TX_\eta)^{-1}X^T_\eta. 
\end{equation}
Since $X_\eta$ consists of pure noise. By \cite{13} the eigenvalues of $X_\eta^TX_\eta$ will converge to the Marc\v{e}nko-Pastur law, with support $[\lambda_-, \lambda_+]$. However, as $N\to\infty$, if $\ell_j > 1+\gamma^{-1}$ by \cite{3} we have $\widehat{\lambda}^{(j)}\to \ell_j+\gamma^{-2}\ell_j/(\ell_j-1)$ which does not lie in the Marc\v{e}nko-Pastur sea $[\lambda_-, \lambda_+]$. The similar statement is true for the smallest few eigenvalues.  Hence we know that almost surely $\det(\widehat{\lambda}^{(j)}I-X_\eta^TX_\eta)\neq 0$. This implies that, without losing generality, we can safely assume that $\det(\widehat{\lambda}^{(j)} I - X_\xi^TA_N(\lambda^{(j)})X_\xi) = 0$. Hence from now on we can restrict our attention on the equation 
\begin{equation}\label{maineqn}
\det(\lambda I - X_\xi^TA_N(\lambda)X_\xi) = 0.
\end{equation} 

As the first observation, from (\ref{eqn2.5}), the matrix $A_N(\lambda)$ is a function of $\lambda$ and $X_\eta$ only. If $\lambda$ is non-random, then $A_N(\lambda)$ is independent of $X_\xi$. Moreover, $X_\xi^TA_N(\lambda)X_\xi$ is a bilinear form of the $\xi_i$'s. In Section \ref{sec4}, we will derive a central limit theorem for such bilinear forms, which will be the core of the whole proof. Indeed, we can prove in subsection \ref{sec2.2} and \ref{sec2.3} that, for fixed $\lambda$, $X_\xi^TA_N(\lambda)X_\xi$ will converge to a diagonal matrix with Gaussian local fluctuations. Based on this, we can get the first and second order approximation of $\widehat{\lambda}^{(j)}$. 

Now let
\begin{eqnarray}
\mathrm{Cov}(\xi_i) & = & \mathrm{diag}\{\ell_1, \ell_2, \ldots, \ell_r\} \nonumber\\
& = & \mathrm{diag}\{\underbrace{\alpha_1, \ldots, \alpha_1}_{r_1}, \underbrace{\alpha_2, \ldots, \alpha_2}_{r_2}, \ldots, \underbrace{\alpha_q, \ldots, \alpha_q}_{r_q}\}
\end{eqnarray}
Here $r_1 + \ldots + r_q = r$. Also for convenience we define the index set
\begin{equation}
I_j := \{i \in\{1, \ldots, r\}: \ell_i = \alpha_j\}, \qquad j = 1, \ldots, q. 
\end{equation}
We are interested in the eigenvalue in the $j$-th pack, namely, $\lambda^{(r_1+\ldots + r_{j-1} + i)}$ for $1\leq i\leq r_j$. If $|\alpha_j-1|>\gamma^{-1}$, then we know almost surely as $N\to\infty$
\begin{equation}
\lambda^{(r_1+\ldots + r_{j-1} + i)} \to \rho_{\alpha_j} := \alpha_j + \frac{\gamma^{-2}\alpha_j}{\alpha_j-1}.
\end{equation}
The local fluctuation is of order $N^{-1/2}$. Hence in our main equation (\ref{maineqn}) we set 
\begin{equation}
\lambda = \rho_{\alpha_j} + \frac{x_j}{\sqrt{N}}. 
\end{equation}
where $x_j = O(1)$ is a non-random number independent of $N$. 

This time 
\begin{equation}\label{part12}
\lambda I - X_\xi^TA_N(\lambda)X_\xi = \biggl(\lambda I - \frac{1}{N}\mathrm{tr}(A_N(\lambda))\widehat{\Sigma}\biggl) - \biggl( X_\xi^TA_N(\lambda)X_\xi - \frac{1}{N}\mathrm{tr}(A_N(\lambda))\widehat{\Sigma}\biggl)
\end{equation}
can be written as the sum of two parts. For the first part,
$\lambda I - \frac{1}{N}\mathrm{tr}(A_N(\lambda))\widehat{\Sigma}$ is a diagonal matrix. For the second part, we shall use Theorem \ref{real clt} to  prove that $\sqrt{N}\cdot(X_\xi^TA_N(\lambda)X_\xi - \frac{1}{N}\mathrm{tr}(A_N(\lambda))\widehat{\Sigma})$ is approximately Gaussian. But first, we need to derive some properties of $A_N(\lambda)$ to ensure that all the conditions in Theorem \ref{real clt} are satisfied.

\subsection{Properties of $A_N(\lambda)$}\label{sec2.1}
In this subsection, we derive some lemmas of our matrix $A_N(\lambda)$, which will be served as some preparation steps for the main theorem. 

First we introduce some constants. Let $[\lambda_-, \lambda_+]$ be the Marc\v{e}nko-Pastur sea. for $1\leq j\leq q$ define
\begin{eqnarray}
m_1(\rho_{\alpha_{j}}) & = & \int_{\lambda_-}^{\lambda_+}\frac{x}{\rho_{\alpha_{j}}-x}F(x)dx= \frac{1}{\alpha_j-1}, \label{m1} \\
m_2(\rho_{\alpha_{j}}, \rho_{\alpha_{j'}}) & = & \int_{\lambda_-}^{\lambda_+}\frac{x^2}{(\rho_{\alpha_{j}}-x)(\rho_{\alpha_{j'}}-x)}F(x)dx \nonumber \\
& = & \frac{(\alpha_{j}-1)(\alpha_{j'}-1) + \gamma^{-2}\alpha_{j}\alpha_{j'}-\gamma^{-2}}{(\alpha_{j}-1)(\alpha_{j'}-1)((\alpha_{j}-1)(\alpha_{j'}-1)-\gamma^{-2})}\label{m2}\\
m_3(\rho_{\alpha_{j}}) & = & \int_{\lambda_-}^{\lambda_+}\frac{x}{(\rho_{\alpha_{j}}-x)^2}F(x)dx = \frac{1}{(\alpha_{j}-1)^2-\gamma^{-2}}.\label{m3} 
\end{eqnarray} 
Here $F(x)$ is the density of the Marc\v{e}nko-Pastur law. 

%
%
\begin{lemma}\label{property1}
Let $A_N(\lambda) = \{a^{(N)}_{st}(\lambda)\}_{s, t=1}^N$, where $\lambda = \rho_{\alpha_{j}} + x_j/\sqrt{N}$ for some fixed $x_j$. We have
\begin{equation}
\frac{1}{N}\sum_{s=1}^Na^{(N)}_{ss}(\lambda)a^{(N)}_{ss}(\lambda') \stackrel{p}{\to} \biggl(1+\frac{\gamma^{-2}[1+m_1(\rho_{\alpha_{j}})]}{\rho_{\alpha_{j}}-\gamma^{-2}[1+m_1(\rho_{\alpha_{j}})]}\biggl)\biggl(1+\frac{\gamma^{-2}[1+m_1(\rho_{\alpha_{j'}})]}{\rho_{\alpha_{j'}}-\gamma^{-2}[1+m_1(\rho_{\alpha_{j'}})]}\biggl). 
\end{equation}
\end{lemma}

%
%
\begin{proof}[Proof for Lemma \ref{property1}]
The proof will be similar to Lemma 6.1 in \cite{4}. By that lemma we have for each $1\leq s\leq N$,
\begin{equation}
a^{(N)}_{ss}(\lambda) \stackrel{p}{\to} 1+\frac{\gamma^{-2}[1+m_1(\rho_{\alpha_{j}})]}{\rho_{\alpha_{j}}-\gamma^{-2}[1+m_1(\rho_{\alpha_{j}})]} := C_j 
\end{equation}
where $C_j$ serves as a shorthand notation. Also we have for each $\lambda$, by Lemma 6.1 in \cite{4},
\[
\sup_N\mathbb{E}\Big(a^{(N)}_{11}(\lambda)\Big)^4 < \infty. 
\]
Hence for $\lambda$ and $\lambda'$ and any $M>0$ we have 
\[
\mathbb{E}\Big[|a^{(N)}_{11}(\lambda)a^{(N)}_{ii}(\lambda')|\mathbbm{1}_{|a^{(N)}_{11}(\lambda)a^{(N)}_{ii}(\lambda')|> M}\Big] \leq \frac{1}{2M}\biggl[\sup_N\mathbb{E}\Big(a^{(N)}_{11}(\lambda)\Big)^4+\sup_N\mathbb{E}\Big(a^{(N)}_{11}(\lambda')\Big)^4\biggl]
\]
Hence $a^{(N)}_{11}(\lambda)a^{(N)}_{11}(\lambda')$ is uniformly integrable. Hence 
\begin{eqnarray*}
&&\mathbb{E}\biggl|\frac{1}{N}\sum_{s=1}^Na^{(N)}_{ss}(\lambda)a^{(N)}_{ss}(\lambda')-C_jC_{j'}\biggl| 
\leq\mathbb{E}\Big|a^{(N)}_{11}(\lambda)a^{(N)}_{11}(\lambda')-C_jC_{j'}\Big| 
\to 0. 
\end{eqnarray*}
\end{proof}

%
%
\begin{lemma}\label{property2}
For the same setting in Lemma \ref{property1} we have
\begin{equation}
\frac{1}{N}\sum_{s, t=1}^Na_{st}^{(N)}(\lambda)a_{st}^{(N)}(\lambda') \stackrel{p}{\to} 1+\gamma^{-2}m_1(\rho_{\alpha_{j}}) + \gamma^{-2}m_1(\rho_{\alpha_{j'}}) + \gamma^{-2}m_2(\alpha_{j}, \rho_{\alpha_{j'}}).
\end{equation}
\end{lemma}

%
%
\begin{proof}[Proof for Lemma \ref{property2}]
We have $A_N(\lambda) = I+B_N(\lambda)$ for $B_N(\lambda) := X_\eta(\lambda I - X_\eta^TX_\eta)^{-1}X_\eta^T$. By Lemma 6.1 in \cite{4} we have 
\begin{eqnarray*}
\frac{1}{N}\mathrm{tr}(B_N(\lambda)) & \stackrel{p}{\to} & \gamma^{-2}m_1(\rho_{\alpha_{j}}), \\
\frac{1}{N}\mathrm{tr}(B_N^T(\lambda)B_N(\lambda)) & \stackrel{p}{\to} & \gamma^{-2}m_2(\rho_{\alpha_{j}}, \rho_{\alpha_{j'}}).
\end{eqnarray*}
Hence 
\begin{eqnarray*}
\frac{1}{N}\sum_{s, t=1}^Na_{st}^{(N)}(\lambda)a_{uv}^{(N)}(\lambda') & = & \frac{1}{N}\mathrm{tr}(A_N^T(\lambda)A_N(\lambda')) \\
& = & 1+\frac{1}{N}\mathrm{tr}(B_N(\lambda))+\frac{1}{N}\mathrm{tr}(B_N(\lambda')) + \frac{1}{N}\mathrm{tr}(B_N^T(\lambda')B_N(\lambda)) \\
& \stackrel{p}{\to} & 1+\gamma^{-2}m_1(\rho_{\alpha_{j}}) + \gamma^{-2}m_1(\rho_{\alpha_{j'}}) + \gamma^{-2}m_2(\rho_{\alpha_{j}}, \rho_{\alpha_{j'}}).
\end{eqnarray*}
\end{proof}

%
%
\begin{lemma}\label{property3}
For the same setting in Lemma \ref{property1},
there exists some constants $M>0$ and $c>0$ such that
\begin{equation}
\mathbb{P}\biggl(\max_{s, t=1}^N|a^{(N)}_{st}(\lambda)|>M\biggl) \leq \exp(-cN), \qquad \text{ for all } N\geq1. 
\end{equation}
\end{lemma}

%
%
\begin{proof}[Proof for Lemma \ref{property3}]
Following the notation in Lemma \ref{property2} we still set $A_N(\lambda) = I+B_N(\lambda)$. By the singular value decomposition of $B_N$ we know
\[
\|B_N(\lambda)\|_2 \stackrel{p}{\to} \max\biggl(\frac{\lambda_+^2}{(\lambda_+-\rho_{\alpha_{j}})^2}, \frac{\lambda_-^2}{(\lambda_--\rho_{\alpha_{j}})^2}\biggl).
\]
Also, by the limiting distribution of extreme eigenvalues of we have, for any \\$M > \max({\lambda_+^2}/{(\lambda_+-\rho_{\alpha_{j}})^2}, {\lambda_-^2}/{(\lambda_--\rho_{\alpha_{j}})^2})$, there exists some constant $c>0$ such that 
\[
\mathbb{P}(\|B_N(\lambda)\|_2>M) \leq \exp(-cN). 
\]
For any $s, t \in \{1, 2, \ldots, N\}$, denote $e_s = (0, \ldots, 1, \ldots, 0)^T$ to be the column vector with all zero entries except a single one in the $s$-th entry. Hence
\[
|a^{(N)}_{st}(\lambda)| \leq |(B_N(\lambda))_{st}|+1\leq \|e_sB_N(\lambda)e_t\|_2+1 \leq \|B_N(\lambda)\|_2+1. 
\]
This gives 
\[
\mathbb{P}\biggl(\max_{s, t=1}^N|a^{(N)}_{st}(\lambda)|>M+1\biggl)\leq \mathbb{P}(\|B_N(\lambda)\|_2>M) \leq \exp(-cN). 
\]
\end{proof}

Lemma \ref{property1} --- Lemma \ref{property3} imply that for $\lambda = \rho_{\alpha_{j}}+x_{j}/\sqrt{N}$ with fixed $x_j$, $A_N(\lambda)$ satisfies all the assumptions in  Theorem \ref{real clt}.
Equipped with these lemmas, now we can realize our promise --- to establish a central limit theorem for $X_\xi^TA_N(\lambda)X_\xi$ as well as the local fluctuations of $\widehat{\lambda}^{(j)}$. This is to be done in the next two subsections.

\subsection{Formula for the $j$-th pack of the sample eigenvalues}\label{sec2.2}

Recall in (\ref{part12}) we decomposed $\lambda I-X_\xi^TA_N(\lambda)X_\xi$ into two parts. By Corollary \ref{point}, the second part $X_\xi^TA_N(\lambda)X_\xi - \mathrm{tr}(A_N(\lambda))\widehat{\Sigma}/N\to 0$ in probability. Let's find the limit for the first part $\lambda I - \mathrm{tr}(A_N(\lambda))\widehat{\Sigma}$. This is a diagonal matrix and for $s\in \{1, 2\ldots, r\}$, if $s\in I_i$ for some $i\neq j$ we have
\begin{eqnarray*}
\Big(\lambda I - \mathrm{tr}(A_N(\lambda))\widehat{\Sigma}\Big)_{ss} & = & \rho_{\alpha_j} - (1+\gamma^{-2}m_1(\rho_{\alpha_j}))\alpha_i + N^{-1/2} \\
& = & \rho_{\alpha_j} - \rho_{\alpha_j}\frac{\alpha_i}{\alpha_j} + N^{-1/2}. 
\end{eqnarray*}
Here I used the equality $\rho_{\alpha_j} = (1+m_1(\rho_{\alpha_j}){\alpha_j}))\alpha_j$ for any $j =1, 2, \ldots, q$. For $s\in I_j$, we have 
\begin{eqnarray}
\Big(\lambda I - \mathrm{tr}(A_N(\lambda))\widehat{\Sigma}\Big)_{ss} & = & \rho_{\alpha_j} + \frac{x_j}{\sqrt{N}} - (1+\gamma^{-2}m_1(\rho_{\alpha_j} + x_j/\sqrt{N}))\alpha_j + o_p(N^{-1/2}) \nonumber\\
& = & \frac{1}{\sqrt{N}}\Big(1+\gamma^{-2}\alpha_jm_3(\rho_{\alpha_j})\Big)x_j + o_p(N^{-1/2}). \label{x}
\end{eqnarray}
In summary, by Corollary \ref{point} for any $0<\kappa<1/2$ we have
\begin{equation}\label{kappa}
\Big(\lambda I - \mathrm{tr}(A_N(\lambda))\widehat{\Sigma}\Big)_{st} = \left\{
\begin{array}{ll}
(\rho_{\alpha_j}-\rho_{\alpha_j}\frac{\alpha_i}{\alpha_j})\mathbbm{1}_{s=t} + o_p(N^{-\kappa}) & \text{ if } s\in I_i \text{ for some } i\neq j \\
o_p(N^{-\kappa}) & \text{ if } s\in I_j
\end{array}
\right.
\end{equation}

Hence the matrix $\lambda I-X_\xi^TA_N(\lambda)X_\xi$ will converge to a diagonal matrix with the block $I_j\times I_j$ being all zeros. This is quite intuitive as $\rho_j$ is the limit for the $j$-th pack of the sample eigenvalues.

By analyzing the limit of $\lambda I-X_\xi^TA_N(\lambda)X_\xi$, we can only obtain the first order approximation of the $j$-th pack of the sample eigenvalues. In order to get the second order approximation of the $j$-th pack, we need to obtain the second order approximation of the matrix $\lambda I-X_\xi^TA_N(\lambda)X_\xi$. Thus we define the matrix $G^{(j)}\in\mathbb{R}^{r\times r}$ such that
\begin{equation}
\Big(G^{(j)}\Big)_{st} = \left\{\begin{array}{ll}\Big(\lambda I-X_\xi^TA_N(\lambda)X_\xi\Big)_{st} & \text{ if } s\notin I_j \\ \sqrt{N}\Big(\lambda I-X_\xi^TA_N(\lambda)X_\xi\Big)_{st} & \text{ if } s\in I_j\end{array}\right.
\end{equation}
That is, we define $G^{(j)}$ by multiplying the rows $I_j$ of $\lambda I-X_\xi^TA_N(\lambda)X_\xi$ by $\sqrt{N}$, leaving the rest of the rows unchanged. 

Since $\det(G^{(j)}) = \det(\lambda I  - X_\xi^TA_N(\lambda)X_\xi)\cdot N^{-r_j/2}$, in order to get the limiting behavior of the $j$-th pack of sample eigenvalues, we can turn to analyze the roots of the equation $\det(G^{(j)}) = 0.$ We know that its rows indexed by $\{1, \ldots, r\}\backslash I_j$ are asymptotically diagonal. For the rows indexed by $I_j$, they will be dense. By our central limit theorem in Section \ref{sec4}, they will be of order $O_p(1)$. Hence intuitively regarding the determinant of $G^{(j)}$ we have the following lemma.

%
%
\begin{lemma}\label{det}
If $\lambda = \rho_{\alpha_j} + x_j/\sqrt{N}$ for $x_j$ fixed, we have 
\begin{equation}
\det G^{(j)} = \det\Big([G^{(j)}]_{I_j\times I_j}\Big)\cdot \prod_{\stackrel{i = 1}{i\neq j}}^q\biggl(\rho_{\alpha_j} - \rho_{\alpha_j}\frac{\alpha_i}{\alpha_j}\biggl)^{r_i} + o_p(1). 
\end{equation}
Here recall $r_i = |I_i|$ and $[G^{(j)}]_{I_j\times I_j}$ is the sub-matrix of $G^{(j)}$ with rows and columns indexed by $I_j$. 
\end{lemma}

\begin{proof}[Proof for Lemma \ref{det}]
By expanding the determinant we have 
\[
\det G^{(j)} = N^{r_i/2}\cdot \sum_{\sigma} \mathrm{sgn}(\sigma)\prod_{s=1}^{r}\Big(\lambda I-X_\xi^TA_N(\lambda)X_\xi\Big)_{s, \sigma(s)}
\]
As in \cite{8} we just need to prove, for any permutation $\sigma$ such that there exists some $s_0\notin I_j, \sigma(s_0)\neq s_0$, we always have
\begin{equation}\label{p0}
N^{r_i/2}\prod_{s=1}^{r}\Big(\lambda I-X_\xi^TA_N(\lambda)X_\xi\Big)_{s, \sigma(s)} \stackrel{p}{\to} 0. 
\end{equation} 
Indeed, for any $0<\kappa <1/2$, if $s\in I_j$ we must have, by (\ref{kappa}), 
\[
\Big(\lambda I-X_\xi^TA_N(\lambda)X_\xi\Big)_{s, \sigma(s)} = o_p(N^{-\kappa}). 
\]
If there further  exists some $s_0\notin I_j$ such that $s_0\neq \sigma(s_0)$, then by (\ref{kappa}) we have 
\[
\Big(\lambda I-X_\xi^TA_N(\lambda)X_\xi\Big)_{s_0, \sigma(s_0)} = o_p(N^{-\kappa}).
\]
Hence 
\[
N^{r_i/2}\prod_{s=1}^{r}\Big(\lambda I-X_\xi^TA_N(\lambda)X_\xi\Big)_{s, \sigma(s)} \leq O_p(N^{r_i/2 - r_i\kappa - \kappa}) = O_p(N^{r_i/2 - (r_i+1)\kappa}). 
\]
Thus we just need to choose $\kappa > r_i/[2(r_i+1)] $ to prove (\ref{p0}). This completes the proof of the lemma. 
\end{proof}

By Lemma \ref{det}, in order to get the asymptotic behavior of the $j$-th pack of the eigenvalues we just need to consider the $r_j$ roots of the equation
\[
\det \Big([G^{(j)}]_{I_j\times I_j}\Big) = 0.  
\]
By (\ref{x}) we have for fixed $x_j$,
\begin{equation}
G^{(j)}_{I_j\times I_j} = \Big(1+\gamma^{-2}\alpha_jm_3(\rho_{\alpha_j})\Big)x_jI- R^{(j)}
\end{equation}
where 
\begin{eqnarray}
R^{(j)} & = & \sqrt{N}\biggl( X_\xi^TA_N(\lambda)X_\xi - \frac{1}{N}\mathrm{tr}(A_N(\lambda))\widehat{\Sigma}\biggl)_{I_j\times I_j} \nonumber\\
& = & \sqrt{N}\biggl( (X_\xi[: , I_j])^TA_N(\lambda)X_\xi[: , I_j] - \frac{1}{N}\mathrm{tr}(A_N(\lambda))\alpha_j I\biggl).
\end{eqnarray}
Here $X_\xi[:, I_j]$ represents the sub matrix of $X_\xi$ consisting only the columns indexed by $I_j$. 

Later we will get a central limit theorem for $R^{(j)}$ for fixed $x_j$. That is, we will prove that $R^{(j)}$ will converge weakly to a Gaussian matrix. Hence intuitively speaking the limiting distribution of the $j$-th pack of the eigenvalues will be the same as the $r_j$ eigenvalues of certain Gaussian matrix $R^{(j)}/(1+\gamma^{-2}\alpha_jm_3(\rho_{\alpha_j}))$. More rigorous proofs will be provided in the next subsection.  

Further, in order to get the joint distribution of these $q$ packs of eigenvalues, it suffices to establish the joint distribution of these $q$ Gaussian matrices $R^{(j)}$ for $j = 1, \ldots, q$. Since they are Gaussian, we just need to characterize the covariance between different entries among these $G^{(j)}$'s. This will be done in the next subsection.

\subsection{Central limit theorem for $\{G^{(j)}\}_{j=1}^q$ and finishing the proof}\label{sec2.3}

In this subsection we apply Theorem \ref{real clt} for $\{R^{(j)}\}_{j=1}^q$. Using the notation of Theorem \ref{real clt}, we have $K = \sum_{i=1}^q r_i(r_i+1)/2$. For $\ell = 1, 2, \ldots, K$, we have 
\[
A_N(\ell) = I + X_\eta\biggl(\biggl(\rho_{\alpha_j}+\frac{x_j}{\sqrt{N}}\biggl)I - X_\eta^TX_\eta\biggl)^{-1}X_\eta^T,
\qquad\forall \sum_{i=1}^{j-1}\frac{r_i(r_i+1)}{2}+1\leq\ell \leq\sum_{i=1}^{j}\frac{r_i(r_i+1)}{2}.
\]
By Lemma \ref{property1} --- \ref{property3} we know that the $A_N(\ell)$ satisfy the assumptions of Theorem \ref{real clt}, with $\omega_{\ell\ell'}$ and $\theta_{\ell\ell'}$ defined by below. For any $\ell, \ell '$ such that 
\[
\sum_{i=1}^{j-1}\frac{r_i(r_i+1)}{2}+1\leq\ell \leq\sum_{i=1}^{j}\frac{r_i(r_i+1)}{2}, \qquad \sum_{i=1}^{j'-1}\frac{r_i(r_i+1)}{2}+1\leq\ell' \leq\sum_{i=1}^{j'}\frac{r_i(r_i+1)}{2}
\]
we have
\begin{eqnarray}
\omega_{\ell\ell'} & = & \biggl[1+\frac{\gamma^{-2}[1+m_1(\rho_{\alpha_{j}})]}{\rho_{\alpha_{j}}-\gamma^{-2}[1+m_1(\rho_{\alpha_{j}})]}\biggl]\biggl[1+\frac{\gamma^{-2}[1+m_1(\rho_{\alpha_{j'}})]}{\rho_{\alpha_{j'}}-\gamma^{-2}[1+m_1(\rho_{\alpha_{j'}})]}\biggl] \nonumber\\
& = & \biggl(1+\frac{\gamma^{-2}}{\alpha_j-1}\biggl)\biggl(1+\frac{\gamma^{-2}}{\alpha_{j'}-1}\biggl). := \widetilde{\omega}_{jj'} \label{real1}\\
\theta_{\ell\ell'} & = & 1+\gamma^{-2}m_1(\rho_{\alpha_{j}}) + \gamma^{-2}m_1(\rho_{\alpha_{j'}}) + \gamma^{-2}m_2(\rho_{\alpha_{j}}, \rho_{\alpha_{j'}})  \nonumber \\
& = & \frac{(\alpha_{j}-1+\gamma^{-2})(\alpha_{j'}-1+\gamma^{-2})}{(\alpha_j-1)(\alpha_{j'}-1)-\gamma^{-2}} := \widetilde{\theta}_{jj'}\label{real2}
\end{eqnarray}
For notational convenience, we denote the right hand side of (\ref{real1}) and (\ref{real2}) to be $\widetilde{\omega}_{jj'}$ and $\widetilde{\theta}_{jj'}$, respectively.

Also we have 
\begin{eqnarray*}
&& \biggl\{X\biggl(\sum_{i=1}^{j-1}\frac{r_i(r_i+1)}{2}+\ell\biggl)\biggl\}_{\ell=1}^{r_j(r_j+1)/2} \\
& = & \Big\{\underbrace{\xi_{\sum_{i=1}^{j-1}r_i+1}, \ldots, \xi_{\sum_{i=1}^{j-1}r_i+1}}_{r_j}, \underbrace{\xi_{\sum_{i=1}^{j-1}r_i+2}, \ldots, \xi_{\sum_{i=1}^{j-1}r_i+2}}_{r_j-1}, \ldots, \underbrace{\xi_{\sum_{i=1}^{j-1}r_i+r_j}}_{1}\Big\}\\
&& \biggl\{Y\biggl(\sum_{i=1}^{j-1}\frac{r_i(r_i+1)}{2}+\ell\biggl)\biggl\}_{\ell=1}^{r_j(r_j+1)/2} \\ 
& = & \Big\{\underbrace{\xi_{\sum_{i=1}^{j-1}r_i+1}, \ldots, \xi_{\sum_{i=1}^{j-1}r_i+r_j}}_{r_j}, \underbrace{\xi_{\sum_{i=1}^{j-1}r_i+2}, \ldots, \xi_{\sum_{i=1}^{j-1}r_i+r_j}}_{r_j-1}, \ldots, \underbrace{\xi_{\sum_{i=1}^{j-1}r_i+r_j}}_{1}\Big\}
\end{eqnarray*}

Then by using our Theorem \ref{real clt}, for every fixed $x_j$, our $q$ matrices $\{R^{(j)}\}_{j=1}^q$ will converge to $q$ Gaussian matrices denoted by $\{G^{(j)}\}_{j=1}^q$, with intra-matrix-covariance 
\begin{eqnarray*}
\mathrm{Cov}(G^{(j)}_{st}, G^{(j)}_{uv}) & = & 
\widetilde{\omega}_{jj}\Big[\mathbb{E}[\xi_{\sum_{i=1}^{j-1}r_i + s, 1}\xi_{\sum_{i=1}^{j-1}r_i+u, 1}\xi_{\sum_{i=1}^{j-1}r_i + t, 1}\xi_{\sum_{i=1}^{j-1}r_i + v, 1}] - \alpha_j^2\mathbbm{1}_{s=t, u = v}\Big] \\
&&  + (\widetilde{\theta}_{jj}-\widetilde{\omega}_{jj})\Big[\alpha_j^2\mathbbm{1}_{s=v, u=t} +\alpha_j^2\mathbbm{1}_{s=u, t=v}\Big]
\end{eqnarray*}
for any $1\leq s<t\leq r_j, 1\leq u<v \leq r_j$. For the inter-matrix-covariances, we have
\begin{eqnarray*}
\mathrm{Cov}(G^{(j)}_{st}, G^{(j')}_{uv}) & = & 
\widetilde{\omega}_{jj'}\Big[\mathbb{E}[\xi_{\sum_{i=1}^{j-1}r_i + s, 1}\xi_{\sum_{i=1}^{j'-1}r_i+u, 1}\xi_{\sum_{i=1}^{j-1}r_i + t, 1}\xi_{\sum_{i=1}^{j'-1}r_i + v, 1}] - \alpha_j\alpha_{j'}\mathbbm{1}_{s=t, u = v}\Big]
\end{eqnarray*}
for any $1\leq s, t\leq r_j$, $1\leq u, v\leq r_{j'}$ and for any $1\leq j\neq j'\leq q$. 

We just proved that for fixed $\{x_j\}_{j=1}^q$, $\{R^{(j)}\}_{j=1}^q$ will jointly converge to Gaussian matrices $\{G^{(j)}\}_{j=1}^q$ weakly. Recall that our matrix $R^{(j)} = R^{(j)}(x_j)$ can be regarded as a stochastic process on $x_j\in\mathbb{R}$. Our next lemma study the convergence of $\{R^{(j)}\}_{j=1}^q$ as a process.  

\begin{lemma}\label{process}
The stochastic process $\{R^{(j)}(x_j)\}_{j=1}^q$ defined on $(x_j)_{j=1}^q\in\mathbb{R}^{q}$ converge to $\{G^{(j)}\}_{j=1}^q$ weakly in the sense o finite dimensional distribution. 
\end{lemma}

\begin{proof}[Proof of Lemma \ref{process}]
We just need to prove that the finite dimensional distribution of the process $\{G^{(j)}(x_j)\}_{j=1}^q$ will converge weakly to that of the process $\{G^{(j)}\}_{j=1}^q$ (this is a process constant in $x_j$). That is, we need to prove that for any positive integer $k$ and any $\{x_{j1}\}_{j=1}^q, \ldots, \{x_{jk}\}_{j=1}^q\in\mathbb{R}^q$, the distribution of 
\[
R^{(1)}(x_{11}), \ldots, R^{(1)}(x_{1k}), R^{(2)}(x_{21}), \ldots, R^{(2)}(x_{2k}), \ldots, R^{(q)}(x_{q1}), \ldots, R^{(q)}(x_{qk})
\]
will converge weakly to the distribution of 
\[
G^{(1)}, \ldots, G^{(1)}, G^{(2)}, \ldots, G^{(2)}, \ldots, G^{(q)}, \ldots, G^{(q)}.
\]
The proof is very similar to what we have done before --- just to use Theorem \ref{real clt} for all these $qk$ matrices $\{R^{(j)}(x_{ji})\}_{1\leq j\leq q, 1\leq i\leq k}$. 
\end{proof}

Now putting all the parts together, we can finally finish the proof of Theorem \ref{main theorem real}.

\begin{proof}[Proof of Theorem \ref{main theorem real}]
Now recall that $\{\widehat{\lambda}^{(\ell)}\}_{\ell=1}^{r}$ are our $r$ extreme sample eigenvalues. Denote 
\[
\widehat{\lambda}^{(\ell)} = \rho_{\alpha_j} + \frac{\widehat{x}_\ell}{\sqrt{N}} \qquad \text{ if } \sum_{i=1}^{j-1}r_i + 1\leq \ell\leq\sum_{i=1}^{j}r_i. 
\]
Now these $\widehat{x}_\ell$'s are random, being no longer fixed. Now let 
\[
y_{j, 1} < z_{j, 1} < y_{j, 2} < z_{j, 2} < \ldots < y_{j, r_j} < z_{j, r_j}, \qquad 1\leq j\leq q
\]
be $2\sum_{i=1}^qr_i = 2r$ fixed constants. For notational convenience, define
\[
\phi_{j, \ell} = \rho_{\alpha_j} + \frac{y_{j, \ell}}{\sqrt{N}},\quad \psi_{j, \ell} = \rho_{\alpha_j} + \frac{z_{j, \ell}}{\sqrt{N}}, \quad, \ell = 1, \ldots, r_j, \quad, j = 1, \ldots, q. 
\]

Then 
\begin{eqnarray*}
&& \mathbb{P}\Big[y_{j, \ell} < \widehat{x}_{\sum_{i=1}^{j-1}r_i+\ell} < z_{j, \ell}, \quad \forall \ell = 1, \ldots, r_j, \quad\forall j = 1, \ldots, q\Big] \\
& = & \mathbb{P}\Big[\det\Big(\phi_{j, \ell}-X_\xi^TA_N(\phi_{j, \ell})X_\xi\Big)\det\Big(\psi_{j, \ell}-X_\xi^TA_N(\psi_{j, \ell})X_\xi\Big)<0,  \\
&&\qquad\qquad\qquad\qquad\qquad\qquad\qquad \forall \ell = 1, \ldots, r_j, \quad\forall j = 1, \ldots, q\Big] \\
& \to & \mathbb{P}\Big[\det\Big(y_{j, \ell}I-\frac{G^{(j)}}{1+\gamma^{-2}\alpha_jm_3(\rho_{\alpha_j})}\Big)\det\Big(z_{j, \ell}I-\frac{G^{(j)}}{1+\gamma^{-2}\alpha_jm_3(\rho_{\alpha_j})}\Big)<0,  \\
&&\qquad\qquad\qquad\qquad\qquad\qquad\qquad \forall \ell = 1, \ldots, r_j, \quad\forall j = 1, \ldots, q\Big]
\end{eqnarray*}
This finishes the proof as the last expression is exactly the probability that the $r_j$ eigenvalues of $(1+\gamma^{-2}\alpha_jm_3(\alpha_j))\cdot G^{(j)}$ are between $y_{j, \ell}$ and $z_{j, \ell}$, respectively for $\ell = 1, \ldots, r_j$ and $j = 1, \ldots, q$. 
\end{proof}


\section{Asymptotic result for eigenvectors}\label{sec3}
As we have said in Section \ref{sec1}, in this section we only consider the case where all the true eigenvalues are district. That is, $q=r$ and $\widehat{\Sigma} = \mathrm{diag}\{\alpha_1, \ldots, \alpha_r\}$. 

We denote $(\widehat{u}^T, \widehat{v}^T)^T$ as the eigenvector of the sample covariance matrix, where $\widehat{u}\in\mathbb{R}^{r}, \widehat{v}\in\mathbb{R}^{p-r}$. Then we have
\begin{equation}\label{3.1}
\left(
\begin{array}{cc}
\widehat{\lambda}I- X_\xi^TX_\xi & -X_\xi^TX_\eta \\
-X_\eta^TX_\xi & \widehat{\lambda}I - X_\eta^TX_\eta
\end{array}
\right)
\left(
\begin{array}{c}
\widehat{u}\\
\widehat{v}
\end{array}
\right) = 0. 
\end{equation}
Here $\widehat{\lambda}$ is the corresponding sample eigenvalue. Just the same as in Section \ref{sec2}, since we are only interested in the isolated eigenvalues, we can assume that $\widehat{\lambda}I-X_\eta^TX_\eta$ is non-singular. Then from (\ref{3.1}) we can get
\begin{eqnarray}
\widehat{\lambda}\widehat{u} & = & X_\xi^TA_N(\widehat{\lambda})X_\xi\widehat{u} \label{u}\\
\widehat{v} & = & (\widehat{\lambda}I-X_\eta^TX_\eta)^{-1}X_\eta^TX_\xi\widehat{u}. \label{v}
\end{eqnarray}
In the following few subsections, we will analyze the the behavior of the eigenvector of the $j$-th eigenvalue, i.e., when $\widehat{\lambda}\approx\rho_{\alpha_j}$. We denote such eigenvectors by $(\widehat{u}^{(j)^T}, \widehat{v}^{(j)^T})^T$ and the corresponding eigenvalue by $\widehat{\lambda}^{(j)}$. Since the eigenvectors are unique up to scaling, we require that $\|\widehat{u}^{(j)}\|_2 = 1$ and $\widehat{u}^{(j)}_j \geq 0$. Here $\widehat{u}^{(j)}_j$ is the $j$-th entry of $\widehat{u}^{(j)}$. Also, for notational convenience, we denote $\widehat{u}^{(j)}_{-j}$ as the $(r-1)$-dimensional vector obtained by deleting the $j$-th entry of $\widehat{u}^{(j)}$.

Recall that in Section \ref{sec2} we proved that $\sqrt{N}\cdot(X_\xi^TA_N(\lambda)X_\xi-\mathrm{tr}(A_N(\lambda))/N)$ will converge weakly to a Gaussian matrix, for $\lambda = \rho_{\alpha_j} + x_j/\sqrt{N}$ with fixed $x_j$. In this section, however, we have to deal with $A_N(\widehat{\lambda}^{(j)})$ where $\widehat{x}_{j}$ is a random variable being bounded in probability. Thus we need to have a more generalized result, stated in following lemma.

\begin{lemma}\label{randomA}
For $A_N(\widehat{\lambda}^{(j)})$ defined above we still have 
\[
\sqrt{N}\cdot\biggl(X_\xi^TA_N(\widehat{\lambda}^{(j)})X_\xi-\frac{1}{N}\mathrm{tr}(A_N(\widehat{\lambda}^{(j)}))\biggl) \to G^{(j)}. 
\]
Here $G^{(j)}$ is defined in Theorem \ref{main theorem real}.
\end{lemma}

\begin{proof}[Proof of Lemma \ref{randomA}]
We have 
\[
X_\xi^TA_N(\widehat{\lambda}^{(j)})X_\xi-\frac{1}{N}\mathrm{tr}(A_N(\widehat{\lambda}^{(j)})) 
= X_\xi^TA_N(\rho_{\alpha_j})X_\xi-\frac{1}{N}\mathrm{tr}(A_N(\rho_{\alpha_j}))\widehat{\Sigma} + R_N
\]
Here the first term is asymptotically $G^{(j)}/\sqrt{N}$ by the previous proof. Hence we just need to show that the residual term $R_N = o_P(N^{-1/2})$.   Now $R_N = X_\xi^TD_N(\widehat{\lambda}^{(j)})X_\xi-\frac{1}{N}\mathrm{tr}(D_N(\widehat{\lambda}^{(j)}))\widehat{\Sigma}$ where 
\begin{eqnarray*}
D_N(\widehat{\lambda}^{(j)}) & = & A_N(\widehat{\lambda}^{(j)}) - A_N(\rho_{\alpha_j}) \\
& = & \frac{\widehat{x}_j}{\sqrt{N}}\cdot\biggl(I + X_\eta(\widehat{\lambda}^{(j)}I-X_\eta^TX_\eta)^{-1}(\rho_{\alpha_j}I-X_\eta^TX_\eta)^{-1}X_\eta^T\biggl) \\
& = & \underbrace{\frac{\widehat{x}_j}{\sqrt{N}}\cdot\biggl(I + X_\eta(\rho_{\alpha_j}I-X_\eta^TX_\eta)^{-2}X_\eta^T\biggl)}_{\text{part I}} \\
&& \qquad - \underbrace{\frac{\widehat{x}_j^2}{N}\cdot X_\eta(\widehat{\lambda}^{(j)}I-X_\eta^TX_\eta)^{-1}(\rho_{\alpha_j}I-X_\eta^TX_\eta)^{-2}X_\eta^T}_{\text{part II}}
\end{eqnarray*}
For the first part, using Corollary \ref{point} we can show 
\[
\frac{\widehat{x}_j}{\sqrt{N}}\cdot X_\xi^T\biggl(I + X_\eta(\rho_{\alpha_j}I-X_\eta^TX_\eta)^{-2}X_\eta^T\biggl)X_\xi - \frac{1}{N}\mathrm{tr}(D_N(\widehat{\lambda}^{(j)}))\widehat{\Sigma} = o_p(N^{-1/2}). 
\]
For the second part, the norm $\|X_\eta(\widehat{\lambda}^{(j)}I-X_\eta^TX_\eta)^{-1}(\rho_{\alpha_j}I-X_\eta^TX_\eta)^{-2}X_\eta^T\|_2 = O_p(1)$ is bounded in probability. Hence the second part is of order $O_p(N^{-1})$. Putting these things together, we proved that $R_N(\widehat{\lambda}^{(j)}) = o_p(N^{-1/2})$, which completes the proof of the lemma. 
\end{proof}

\subsection{Proof of Theorem \ref{thm2}}\label{sec3.1}
Intuitively, $\widehat{u}^{(j)}$ should be close to $e_j$. Here $e_j$ is the vector of all zeros except the $j$-th entry being one. Hence $\widehat{u}^{(j)}_{j}\approx1$ and $\widehat{u}^{(j)}_{-j}\approx0$. The following lemma establishes out intuition.  

%
%
\begin{lemma}\label{error bound}
We have 
\begin{eqnarray}
\widehat{u}^{(j)}_{j} & = & O_p(N^{-1}) \label{1c}\\
\widehat{u}^{(j)}_{-j} & = & O_p(N^{-1/2}). \label{2c}
\end{eqnarray}
\end{lemma}

%
%
\begin{proof}[proof of Lemma \ref{error bound}]
As is shown in Section \ref{sec2}, we have 
\begin{equation}\label{xax}
X_\xi^TA_N(\widehat{\lambda}^{(j)})X_\xi = \frac{1}{N}\mathrm{tr}\Big(A_N(\widehat{\lambda}^{(j)})\Big)\widehat{\Sigma} + \frac{1}{\sqrt{N}}R_N^{(j)}.
\end{equation}
Here 
\begin{equation}
R_N^{(j)} := \sqrt{N}\biggl\{X_\xi^TA_N(\widehat{\lambda}^{(j)})X_\xi - \frac{1}{N}\mathrm{tr}\Big(A_N(\widehat{\lambda}^{(j)})\Big)\widehat{\Sigma}\biggl\}\label{RN}
\end{equation}
will, by Lemma \ref{randomA}, $R_N^{(j)}$ will converge in distribution to a Gaussian matrix. Hence $R_N^{(j)} = O_p(1)$. Moreover by the previous section we also have 
\begin{equation}\label{lambdax}
\widehat{\lambda}^{(j)} = \rho_{\alpha_j} + \frac{1}{\sqrt{N}}\widehat{x}_j
\end{equation}
for some $\widehat{x}_j = O_p(1)$. Substituting (\ref{xax}) and (\ref{lambdax}) in (\ref{u}) gives
\[
\rho_{\alpha_j}\widehat{u}^{(j)} + \frac{\widehat{x}_j}{\sqrt{N}}\widehat{u}^{(j)} = \frac{1}{N}\mathrm{tr}\Big(A_N(\widehat{\lambda}^{(j)})\Big)\widehat{\Sigma}\widehat{u}^{(j)} + \frac{1}{\sqrt{N}}R_N^{(j)}\widehat{u}^{(j)}. 
\]
Since 
\[
\frac{1}{N}\mathrm{tr}(A_N(\widehat{\lambda}^{(j)})) = \frac{\rho_{\alpha_j}}{\alpha_j} - \frac{1}{\sqrt{N}}\gamma^{-2}m_3(\rho_{\alpha_j})\widehat{x}_j + o_p(N^{-1/2})
\]
we obtain 
\begin{equation}\label{better}
\rho_{\alpha_j}\widehat{u}^{(j)} + \frac{\widehat{x}_j}{\sqrt{N}}\widehat{u}^{(j)} = \frac{\rho_{\alpha_j}}{\alpha_j}\widehat{\Sigma}\widehat{u}^{(j)} - \frac{1}{\sqrt{N}}\gamma^{-2}m_3(\rho_{\alpha_j})\widehat{x}_j\widehat{\Sigma}\widehat{u}^{(j)} + \frac{1}{\sqrt{N}}R_N^{(j)}\widehat{u}^{(j)} + o_p(N^{-1/2}). 
\end{equation}
In (\ref{better}), compare all the entries except the $j$-th, we get 
\begin{equation}\label{z}
\rho_{\alpha_j}\biggl(I - \frac{1}{\alpha_j}\widehat{\Sigma}_{-j, -j}\biggl)\widehat{u}^{(j)}_{-j} = O_p(N^{-1/2}).
\end{equation}
Here $\widehat{\Sigma}_{-j, -j}$ is the $(j-1)\times(j-1)$ sub-matrix of $\widehat{\Sigma}$ after deleting its $j$-th row and $j$-th column.  All the rest of the terms in (\ref{better}) can be written as $O_p(N^{-1/2})$ because $\widehat{u}^{(j)} = O_p(1)$ as it has unit norm.  In (\ref{z}) since the matrix on the left hand side is non-singular we must have $\widehat{u}^{(j)}_{-j} = O_p(N^{-1/2})$, proving our first claim (\ref{1c}).

For the second claim, recall $\|u^{(j)}\|_2 = 1$, we have  
\[
|u^{(j)}_j| = \sqrt{1-\|u^{(j)}_{-j}\|_2^2} = \sqrt{1-O_p(N^{-1})} = 1 + O_p(N^{-1}). 
\]
Noting that $u^{(j)}_{j}$ is positive, we proved (\ref{2c}). 
\end{proof}

By Lemma \ref{error bound}, we can write $u^{(j)}$ as 
\begin{equation}\label{udu}
u^{(j)} = e_j + \frac{1}{\sqrt{N}}\delta u^{(j)}
\end{equation}
where $\delta u^{(j)}_{j} = o_p(1)$ and $\delta u^{(j)}_{-j} = O_p({1})$. Substituting (\ref{udu}) in (\ref{better}) we can obtain 
\begin{equation}\label{f}
\biggl(\rho_{\alpha_j} - \frac{\rho_{\alpha_j}}{\alpha_j}\widehat{\Sigma}\biggl)\delta u^{(j)} = -\biggl(\widehat{x}_jI + \gamma^{-2}m_3(\rho_{\alpha_j})\widehat{x}_j\widehat{\Sigma}\biggl)e_j + R_N^{(j)}e_j + o_p(1).
\end{equation}
If we consider all the entries of (\ref{f}) except the $j$-th one, we can obtain
\begin{equation}\label{ff}
\biggl(\rho_{\alpha_j} - \frac{\rho_{\alpha_j}}{\alpha_j}\widehat{\Sigma}_{-j, -j}\biggl)\delta u^{(j)}_{-j} = R_N^{(j)}e_j + o_p(1)
\end{equation}

For every $j$, we can get an equation of $\delta u^{(j)}_{-j}$ as in (\ref{ff}). Now, using Theorem \ref{real clt} as well as the technique in Lemma \ref{randomA}, we know that as $N\to\infty$, our $r$ matrices $R^{(1)}_N, \ldots, R^{(r)}_N$ will jointly converge to $r$ matrices denoted by $G^{(1)}, \ldots, G^{(r)}$, with jointly Gaussian entries of mean zero. Their covariance is 
\[
\mathrm{Cov}(G^{(j)}_{st}, G^{(j')}_{uv}) = \widetilde{\omega}_{jj'}\Big[\mathbb{E}[\xi_{s}\xi_{u}\xi_{t}\xi_{v}] - \alpha_s\alpha_u\mathbbm{1}_{s=t, u=v}\Big] + (\widetilde{\theta}_{jj'}-\widetilde{\omega}_{jj'})[\alpha_s\alpha_t\mathbbm{1}_{s=v, u=t} + \alpha_{\alpha_s\alpha_t}\mathbbm{1}_{s=u, t=v}]
\]
where $\widetilde{\omega}_{jj'}$ and $\widetilde{\theta}_{jj'}$ are defined in (\ref{real1}) and (\ref{real2}). 
Note that the $j$-th eigenvalue is just 
\[
\widehat{\lambda}^{(j)} \stackrel{D}{=} \rho_{\alpha_j} + \frac{1}{\sqrt{N}}\cdot\frac{G^{(j)}_{jj}}{1+\gamma^{-2}\alpha_jm_3(\rho_{\alpha_j})} + o_p(N^{-1/2}). 
\]
Together with the expression of
\begin{equation}
\widehat{u}^{(j)}_{-j} \stackrel{D}{=} N^{-1/2}\cdot\biggl(\frac{\alpha_i}{\rho_{\alpha_j}(\alpha_j-\alpha_i)}G^{(j)}_{ij}\biggl)_{1\leq i\leq r, i\neq j} + o_p(N^{-1/2}),
\end{equation}
we complete the proof of Theorem.

%
%

\subsection{Proof of Theorem \ref{thm3}}
In this subsection we analyze the angle between the sample eigenvector $(\widehat{u}^{(j)}, \widehat{v}^{(j)})^T$ and the true eigenvector $(e_j^T, 0^T)^T$. Here we define $\widehat{\beta}^{(j)}\in[0, \pi]$ by
\begin{equation}\label{cos}
\cos\widehat{\beta}^{(j)} = \frac{\widehat{u}_j^{(j)}}{\sqrt{1+\|\widehat{v}^{(j)}\|_2^2}} = \cos\mathrm{angle}\left(\left(\begin{array}{c}\widehat{u}^{(j)} \\ \widehat{v}^{(j)}\end{array}\right), \left(\begin{array}{c}e_j \\ 0\end{array}\right)\right).
\end{equation}

First, for notational convenience, let's define some functions, just as in the previous section. We define 
\begin{eqnarray}
m_4(\rho_{\alpha_j}) & = & \int_{\lambda_-}^{\lambda_+}\frac{2x}{(\rho_{\alpha_{j}}-x)^3}F(x)dx = \frac{2(\alpha_j-1)^3}{((\alpha_j-1)^2-\gamma^{-2})^3}.\label{m4}  \\
m_5(\rho_{\alpha_j}) & = & \int_{\lambda_-}^{\lambda_+}\frac{1}{(\rho_{\alpha_{j}}-x)^2}F(x)dx = \frac{(\alpha_j-1)^2}{(\alpha_j-1+\gamma^{-2})^2((\alpha_j-1)^2-\gamma^{-2})}.\label{m4}  \\
m_6(\rho_{\alpha_j}) & = & \int_{\lambda_-}^{\lambda_+}\frac{1}{\rho_{\alpha_{j}}-x}F(x)dx = \frac{1}{\alpha_j-1+\gamma^{-2}}.\label{m6}  \\
m_7(\rho_{\alpha_{j}}, \rho_{\alpha_{j'}}) & = & \int_{\lambda_-}^{\lambda_+}\frac{x^2}{(\rho_{\alpha_{j}}-x)^2(\rho_{\alpha_{j'}}-x)^2}F(x)dx \label{m7}\\
& = & \frac{(\alpha_j-1)^2(\alpha_{j'}-1)^2((\alpha_{j}-1)(\alpha_{j'}-1)+\gamma^{-2}(\alpha_{j}\alpha_{j'}+\alpha_{j}+\alpha_{j'}-2)+\gamma^{-4})}{((\alpha_{j}-1)^2-\gamma^{-2})((\alpha_{j'}-1)^2-\gamma^{-2})((\alpha_{j}-1)(\alpha_{j'}-1)-\gamma^{-2})^3}.\nonumber \\
m_8(\rho_{\alpha_{j}}, \rho_{\alpha_{j'}}) & = & \int_{\lambda_-}^{\lambda_+}\frac{x^2}{(\rho_{\alpha_{j}}-x)(\rho_{\alpha_{j'}}-x)^2}F(x)dx \label{m8} \\
& = & \frac{(\alpha_{j}-1)(\alpha_{j'}-1)^2+\gamma^{-2}(\alpha_{j'}-1)(\alpha_{j}\alpha_{j'}+\alpha_{j}-2)-\gamma^{-4}}{((\alpha_{j}-1)(\alpha_{j'}-1)-\gamma^{-2})^2((\alpha_{j'}-1)^2-\gamma^{-2})}.\nonumber
\end{eqnarray}
where $F(x)$ is the density for the Marc\v{e}nko-Pastur law. 

By (\ref{v}) we have 
\begin{equation}\label{normv}
\|\widehat{v}^{(j)}\|_2^2 = \Big(\widehat{u}^{(j)}\Big)^T\cdot X_\xi^TC_N(\widehat{\lambda}^{(j)})X_\xi\cdot\Big(\widehat{u}^{(j)}\Big)
\end{equation}
where 
\[
C_N(\widehat{\lambda}^{(j)}) = X_\eta(\widehat{\lambda}^{(j)}I-X^T_\eta X_\eta)^{-2}X_\eta^T.
\]
This time, we have 
\begin{equation}\label{CN}
X_\xi^TC_N(\widehat{\lambda}^{(j)})X_\xi = \frac{1}{N}\mathrm{tr}\Big(C_N(\widehat{\lambda}^{(j)})\Big)\widehat{\Sigma} + \frac{1}{\sqrt{N}}\cdot Q_N^{(j)}.
\end{equation}
where 
\begin{equation}\label{qn}
Q_N^{(j)} := \sqrt{N}\cdot\biggl(X_\xi^TC_N(\widehat{\lambda}^{(j)})X_\xi - \frac{1}{N}\mathrm{tr}\Big(C_N(\widehat{\lambda}^{(j)})\Big)\widehat{\Sigma}\biggl)
\end{equation}
which, by applying Theorem \ref{real clt} and a similar technique in Lemma \ref{randomA}, will converge to a real Gaussian matrix. Furthermore, we have 
\[
\frac{1}{N}\mathrm{tr}\Big(C_N(\widehat{\lambda}^{(j)})\Big) = \gamma^{-2}m_3(\rho_{\alpha_j}) - \frac{\gamma^{-2}\widehat{x}_j}{\sqrt{N}}m_4(\rho_{\alpha_j}) + o_p(N^{-1/2})
\]

Using the notation of the previous subsection, 
\begin{equation}\label{uj}
\widehat{u}^{(j)} = e_j + \frac{1}{\sqrt{N}}\cdot D^{(j)}R^{(j)}_N e_j + o_p(N^{-1/2})
\end{equation}
where $R^{(j)}_N$ is defined in (\ref{RN}) and 
\[
D^{(j)} = \mathrm{diag}\biggl\{\frac{\alpha_1}{\rho_{\alpha_j}(\alpha_j-\alpha_1)}, \ldots, \frac{\alpha_{j-1}}{\rho_{\alpha_j}(\alpha_j-\alpha_{j-1})}, 0, \frac{\alpha_{j+1}}{\rho_{\alpha_j}(\alpha_j-\alpha_{j+1})}, \ldots, \frac{\alpha_r}{\rho_{\alpha_j}(\alpha_j-\alpha_r)}\biggl\}
\]
Substituting (\ref{CN}) and (\ref{uj}) in (\ref{normv}) we obtain 
\begin{equation}
\|\widehat{v}^{(j)}\|_2^2 = \gamma^{-2}m_3(\rho_{\alpha_j})\alpha_j - \frac{1}{\sqrt{N}}\gamma^{-2}m_4(\rho_{\alpha_j})\alpha_j\widehat{x}_j + \frac{1}{\sqrt{N}}e_j^TQ^{(j)}_Ne_j + o_p(N^{-1/2}). 
\end{equation}
Using this formula in (\ref{cos}) we get 
\begin{multline}
\cos\widehat{\beta}^{(j)} = \frac{1}{\sqrt{1+\gamma^{-2}m_3(\rho_{\alpha_j})\alpha_j}} \\ +\frac{1}{\sqrt{N}}\cdot\frac{1}{[1+\gamma^{-2}m_3(\rho_{\alpha_j})\alpha_j]^{3/2}}\Big(\gamma^{-2}m_4(\rho_{\alpha_j})\alpha_j\widehat{x}_j - e_j^TQ^{(j)}_Ne_j\Big)  + o_p(N^{-1/2})\label{cosbeta}
\end{multline}

In order to get the convergence in distribution of $e_j^TQ^{(j)}_Ne_j$, and its relationship on $\widehat{x}_j$, we can use Theorem \ref{real clt}. Before that, we need to derive some properties of the matrix $C_N(\lambda)$. 

\subsection{Properties of $C_N(\lambda)$ and finishing the proof}

Let $C_N(\lambda) := (c_{st}^{(N)}(\lambda))_{s, t=1}^{N}$ then
%
%
\begin{lemma}\label{property1c}
For $\lambda = \rho_{\alpha_j}+x_j/\sqrt{N}$ and $\lambda' = \rho_{\alpha_{j'}}+x_{j'}/\sqrt{N}$ with $x_j, x_{j'}$ being fixed constants, we have 
\begin{equation}\label{pp1}
\frac{1}{N}\sum_{s=1}^Nc^{(N)}_{ss}(\lambda)c^{(N)}_{ss}(\lambda') \stackrel{p}{\to} \frac{\gamma^{-4}m_5(\lambda)m_5(\lambda')}{(1-\gamma^{-2}m_6(\lambda))^2(1-\gamma^{-2}m_6(\lambda'))^2}. 
\end{equation}
\end{lemma}

%
%
\begin{proof}[Proof of Lemma \ref{property1c}]
Define $X_{\eta, -s}$ to be the $(N-1)\times(p-r)$ sub matrix of $X_\eta$ after deleting the $s$-th row. Then by the Sherman-Morrison formula we have 
\begin{eqnarray*}
(\lambda I - X_\eta^TX_\eta)^{-1} & = & (\lambda I - X_{\eta, -s}^TX_{\eta, -s} - \eta_s\eta_s^T/N)^{-1} \\
& = & (\lambda- X_{\eta, -s}^TX_{\eta, -s})^{-1} + \frac{\frac{1}{N}(\lambda- X_{\eta, -s}^TX_{\eta, -s})^{-1}\eta_s\eta_s^T(\lambda- X_{\eta, -s}^TX_{\eta, -s})^{-1}}{1-\frac{1}{N}\eta_s^T(\lambda- X_{\eta, -s}^TX_{\eta, -s})^{-1}\eta_s}.
\end{eqnarray*}
Taking the square and pre(reps., post) multiplying $\eta_s^T$(reps., $\eta_s$) gives
\begin{eqnarray*}
\frac{1}{N}\eta_s^T(\lambda I - X_\eta^TX_\eta)^{-2}\eta_s & = & 
\frac{1}{N}\eta_s^T(\lambda- X_{\eta, -s}^TX_{\eta, -s})^{-2}\eta_s\\
&& + \frac{\frac{1}{N^3}\Big(\eta_s^T(\lambda- X_{\eta, -s}^TX_{\eta, -s})^{-1}\eta_s\Big)^2\Big(\eta_s^T(\lambda- X_{\eta, -s}^TX_{\eta, -s})^{-2}\eta_s\Big)}{\Big(1-\frac{1}{N}\eta_s^T(\lambda- X_{\eta, -s}^TX_{\eta, -s})^{-1}\eta_s\Big)^2} \\
&& + \frac{\frac{2}{N^2}\Big(\eta_s^T(\lambda- X_{\eta, -s}^TX_{\eta, -s})^{-1}\eta_s\Big)\Big(\eta_s^T(\lambda- X_{\eta, -s}^TX_{\eta, -s})^{-2}\eta_s\Big)}{1-\frac{1}{N}\eta_s^T(\lambda- X_{\eta, -s}^TX_{\eta, -s})^{-1}\eta_s}.
\end{eqnarray*}
Using the same proof as that of Lemma 6.1 in \cite{3} we can prove that 
\begin{equation}
c^{(N)}_{ss}(\lambda) = \frac{1}{N}\eta_s^T(\lambda I - X_\eta^TX_\eta)^{-2}\eta_s \stackrel{p}{\to} \frac{\gamma^{-2}m_5(\lambda)}{(1-\gamma^{-2}m_6(\lambda))^2}.  
\end{equation}
Similar to Lemma \ref{property1} we can prove that $c_{11}^{(N)}(\lambda)c^{(N)}_{11}(\lambda')$ is uniformly integrable in $N$. Hence (\ref{pp1}) follows.  
\end{proof}

%
%
\begin{lemma}\label{property2c}
For $\lambda = \rho_{\alpha_j}+x_j/\sqrt{N}$ and $\lambda' = \rho_{\alpha_{j'}}+x_{j'}/\sqrt{N}$ with $x_j, x_{j'}$ being fixed constants, we have
\begin{equation}
\frac{1}{N}\sum_{s, t=1}^Nc_{st}^{(N)}(\lambda)c_{st}^{(N)}(\lambda') \stackrel{p}{\to} \gamma^{-2}m_7(\lambda, \lambda').
\end{equation}
\end{lemma}

\begin{proof}[Proof of Lemma \ref{property2c}]
We have 
\begin{eqnarray*}
\frac{1}{N}\sum_{s, t=1}^Nc_{st}^{(N)}(\lambda)c_{st}^{(N)}(\lambda') & = & \frac{1}{N}\mathrm{tr}C_N(\lambda)C_N^T(\lambda') \\
& = & \frac{1}{N}\mathrm{tr}\Big[X_\eta(\lambda I-X_\eta^TX_\eta)^{-2}X_\eta^TX_\eta(\lambda' I-X_\eta^TX_\eta)^{-2}X_\eta^T\Big] \\
& \stackrel{p}{\to} & \gamma^{-2}m_7(\lambda, \lambda'). 
\end{eqnarray*}
\end{proof}

For completeness, we list the third lemma below. This is very similar to Lemma \ref{property3}, and the proof is almost the same. Hence we omit that. 
%
%
\begin{lemma}\label{property3c}
For $\lambda = \rho_{\alpha_j}+x_j/\sqrt{N}$ with $x_j$ being fixed constants, there exists some constants $M>0$ and $c>0$ such that
\begin{equation}
\mathbb{P}\biggl(\max_{s, t=1}^N|c^{(N)}_{st}(\lambda)|>M\biggl) \leq \exp(-cN). 
\end{equation}
\end{lemma}

Also, regarding the interaction term between the matrix $A_N(\lambda)$ and the $C_N(\lambda')$, we have the following two lemmas. Again, due to the fact that the proof is almost the same, we omit the proof here. 
%
%
\begin{lemma}\label{propertyac}
For $\lambda = \rho_{\alpha_j}+x_j/\sqrt{N}$ and $\lambda' = \rho_{\alpha_{j'}}+x_{j'}/\sqrt{N}$ with $x_j, x_{j'}$ being fixed constants, we have
\begin{eqnarray*}
\frac{1}{N}\sum_{s=1}^Na^{(N)}_{ss}(\lambda)c^{(N)}_{ss}(\rho_{\alpha_{j}}) & \stackrel{p}{\to} & \biggl[1+\frac{\gamma^{-2}[1+m_1(\lambda)]}{\lambda-\gamma^{-2}[1+m_1(\rho_{\alpha_{j}})]}\biggl]\frac{\gamma^{-2}m_5(\rho_{\alpha_{j'}})}{(1-\gamma^{-2}m_6(\rho_{\alpha_{j'}}))^2}. \\
\frac{1}{N}\sum_{s, t=1}^Na_{st}^{(N)}(\lambda)c_{st}^{(N)}(\lambda') & \stackrel{p}{\to} & \gamma^{-2}m_3(\lambda') + \gamma^{-2}m_8(\lambda, \lambda').
\end{eqnarray*}
\end{lemma}

Equipped with Lemma \ref{property1c} --- \ref{property3c} and we use the same technique in Lemma \ref{randomA}, we can now get a central limit theorem for $X_\xi^TC_N(\widehat{\lambda}^{(j)})X_\xi$. 

\begin{lemma}
The entries $\{e_j^TR^{(j)}_{N}e_j, e_j^TQ^{(j)}_{N}e_j\}_{j=1}^r$ will converge in distribution to $\{G^{(j)}_{jj}, H^{(j)}_{jj}\}_{j=1}^r$, where 
\begin{eqnarray}
\mathrm{Cov}(G^{(j)}_{jj}, G^{(j')}_{j'j'}) & = & \widetilde{\omega}_{jj'}[\mathbb{E}\xi_j^2\xi_{j'}^2 - \alpha_{j}\alpha_{j'}] + 2(\widetilde{\theta}_{jj'}-\widetilde{\omega}_{jj'})\alpha_j^2\mathbbm{1}_{j=j'}, \label{gg}\\
\mathrm{Cov}(H^{(j)}_{jj}, H^{(j')}_{j'j'}) & = & \widetilde{\zeta}_{jj'}[\mathbb{E}\xi_j^2\xi_{j'}^2 - \alpha_{j}\alpha_{j'}] + 2(\widetilde{\tau}_{jj'}-\widetilde{\zeta}_{jj'})\alpha_j^2\mathbbm{1}_{j=j'}, \label{hh}\\
\mathrm{Cov}(G^{(j)}_{jj}, H^{(j')}_{j'j'}) & = & \widetilde{\kappa}_{jj'}[\mathbb{E}\xi_j^2\xi_{j'}^2 - \alpha_{j}\alpha_{j'}] + 2(\widetilde{\mu}_{jj'}-\widetilde{\kappa}_{jj'})\alpha_j^2\mathbbm{1}_{j=j'}. \label{gh}\\
\end{eqnarray}

Here $\widetilde{\zeta}_{jj'}, \widetilde{\tau}_{jj'}, \widetilde{\kappa}_{jj'}$ and $\widetilde{\mu}_{jj'}$ are defined such that for all $1\leq j, j'\leq r$, 
\begin{eqnarray}
\widetilde{\zeta}_{jj'} & := & \frac{\gamma^{-4}m_5(\rho_{\alpha_{j}})m_5(\rho_{\alpha_{j'}})}{(1-\gamma^{-2}m_6(\rho_{\alpha_{j}}))^2(1-\gamma^{-2}m_6(\rho_{\alpha_{j'}}))^2} \label{zzeta}\\
\widetilde{\tau}_{jj'} & := & \gamma^{-2}m_7(\rho_{\alpha_{j}}, \rho_{\alpha_{j'}}) \label{ttau}\\
\widetilde{\kappa}_{jj'} & = & \biggl(1+\frac{\gamma^{-2}[1+m_1(\rho_{\alpha_{j}})]}{\rho_{\alpha_{j}}-\gamma^{-2}[1+m_1(\rho_{\alpha_{j}})]}\biggl)\frac{\gamma^{-2}m_5(\rho_{\alpha_{j'}})}{(1-\gamma^{-2}m_6(\rho_{\alpha_{j'}}))^2} \label{kkappa}\\
\widetilde{\mu}_{jj'} & = & \gamma^{-2}m_3(\rho_{\alpha_{j'}}) + \gamma^{-2}m_8(\rho_{\alpha_{j}}, \rho_{\alpha_{j'}}). \label{mmu}
\end{eqnarray}
\end{lemma}

Noting that jointly in distribution, we have 
\begin{eqnarray*}
\widehat{x}^{(j)} & \stackrel{D}{\to} & \frac{1}{1+\gamma^{-2}\alpha_jm_3(\rho_{\alpha_j})}G^{(j)}_{jj}, \\
e_j^TQ^{(j)}e_j & \stackrel{D}{\to} & H^{(j)}_{jj}.
\end{eqnarray*}
Recall the expression for $\cos\widehat{\beta}^{(j)}$ in (\ref{cosbeta}), the proof is complete.

\section{Proof of Central Limit Theorem}\label{sec4}
In this section we prove a central limit theorem for the bilinear form. This is a separate result and can be used as a tool in the rest of the paper.

%
%
\begin{theorem}\label{clt}
Let $A_N{(\ell)} = (a_{ij}^{(N)}(\ell)), \ell = 1, \ldots, K$ be $K$ sequences of $N\times N$ Hermitian matrices such that all the entries are bounded and the following limits exist in probability.
\begin{eqnarray}
\omega_{\ell\ell'} & = & \lim_{N\to\infty}\frac{1}{N}\sum_{u=1}^N a^{(N)}_{uu}(\ell)a^{(N)}_{uu}(\ell'), \label{p1}\\
\theta_{\ell\ell'} & = & \lim_{N\to\infty}\frac{1}{N}\sum_{u, v=1}^Na^{(N)}_{uv}(\ell)\overline{a}^{(N)}_{uv}(\ell')\label{p2}\\
\tau_{\ell\ell'} & = & \lim_{N\to\infty}\frac{1}{N}\sum_{u, v=1}^Na^{(N)}_{uv}(\ell)a^{(N)}_{uv}(\ell') \label{p3}
\end{eqnarray} 
Also assume there exist some constant $M>0$ and $c>0$ such that 
\begin{equation}\label{p4}
\mathbb{P}(\max_{u, v=1}^N|a^{(N)}_{uv}(\ell)|>M)\leq\exp(-cN).  
\end{equation}

Let $(x_i, y_i)_{1\leq i\leq N}$ be a sequence of complex-valued i.i.d. vectors in $\mathbb{C}^{2K}$, independent of $\{A_N(\ell)\}_{\ell=1}^K$. Here $x_i, y_i \in\mathbb{C}^K$ and 
\[
x_i = \left(\begin{array}{c}x_{1i} \\ \vdots \\ x_{Ki}\end{array}\right), y_i = \left(\begin{array}{c}y_{1i} \\ \vdots \\ y_{Ki}\end{array}\right), \qquad X(\ell) = \left(\begin{array}{c}x_{\ell1} \\ \vdots \\ x_{\ell N}\end{array}\right), Y(\ell) = \left(\begin{array}{c}y_{\ell1} \\ \vdots \\ y_{\ell N}\end{array}\right).
\]
Let 
\begin{equation}
\rho(\ell) = \mathbb{E}[\overline{x}_{\ell1}y_{\ell1}]
\end{equation}
and define 
\[
Z_N = (Z_N(\ell))_{\ell = 1}^K, \qquad Z_N(\ell) = \frac{1}{\sqrt{N}}\Big[X(\ell)^*A_NY(\ell) - \rho(\ell)\mathrm{tr}(A_N(\ell))\Big].
\]
Then as $N\to\infty$, $Z_N$ will weakly converge to a Gaussian distributed random vector $W\in\mathbb{C}^{K}$ with moment generating function 
\[
\mathbb{E}e^{c^TW} = \exp\biggl(\frac{1}{2}c^TBc\biggl), \qquad c\in\mathbb{C}^k
\] 
where $B$ is defined by $B = B_1+B_2+B_3$ where 
\begin{eqnarray}
B_1 & = & \biggl(\Big[\mathbb{E}\overline{x}_{\ell1}\overline{x}_{\ell'1}y_{\ell1}y_{\ell'1}-\rho(\ell)\rho(\ell')\Big]\omega_{\ell\ell'}\biggl)_{\ell, \ell' = 1}^K, \\
B_2 & = & \biggl(\mathbb{E}\Big[\overline{x}_{\ell1}y_{\ell'1}\Big]\mathbb{E}\Big[\overline{x}_{\ell'1}y_{\ell1}\Big](\theta_{\ell\ell'}-\omega_{\ell\ell'})\biggl)_{\ell, \ell'=1}^K \\
B_3 & = & \biggl(\mathbb{E}\Big[\overline{x}_{\ell1}\overline{x}_{\ell'1}\Big]\mathbb{E}\Big[y_{\ell1}y_{\ell'1}\Big](\tau_{\ell\ell'}-\omega_{\ell\ell'})\biggl)_{\ell, \ell'=1}^K 
\end{eqnarray}
\end{theorem}

%
%
\begin{proof}[Proof of Theorem \ref{clt}]
First we state that we can assume, without losing generality, that $A_N(\ell)$ is a series of non-random matrices such that
\begin{itemize}
\item
Limit (\ref{p1}) --- (\ref{p3}) holds with convergence in probability replaced by ordinary convergence. 
\item
$\max_{u, v=1}^N|a^{(N)}_{uv}(\ell)|\leq M$ for some constant $M$, uniformly for all $N$. 
\end{itemize}
Indeed, if $A_N(\ell)$ is random, we know that for any subsequence of $A_N(\ell)$ there exists a sub-sub-sequence of $A_N(\ell)$, with (\ref{p1}) --- (\ref{p3}) holds true in ordinary convergence. Also using Borel-Cantelli lemma, we pick the sub-sub-sequence such that all the elements are bounded. We can turn to work on the corresponding sub-sub-sequence $\{Z_N(\ell)\}_{\ell=1}^K$. By conditioning on $X_\eta$, we can treat $A_N(\ell)$ as deterministic matrices. Hence, if we can prove the theorem for deterministic matrices $A_N(\ell)$, then for any subsequence of $Z_N(\ell)$, there exists a sub-sub-sequence of $Z_N(\ell)$ which always converges to the same limit, independent of the sub-sequence chosen. By using this sub-sequence argument, we proved that the theorem also holds true for random sequences $A_N(\ell)$. 

Thus, we can safely assume that $A_N(\ell)$ is deterministic with bounded elements. The rest of the proof will be very similar to that of Theorem 7.1 in \cite{4}. Since it is a generalization of that theorem, here we just point out the major differences between the two. Also, 

Using truncation as in \cite{4}, we can assume without losing generality that there exists a sequence $\epsilon_N\downarrow 0$ that 
\begin{equation}
\|x_i\|_2\vee \|y_i\|_2 \leq \epsilon_N\cdot N^{1/4}, \qquad \forall 1\leq i\leq N. 
\end{equation}

Just as in \cite{4}, we turn to establish the one dimensional central limit theorem for the random variable
\[
\sum_{\ell=1}^Kc_\ell X(\ell)^*A_N(\ell)Y(\ell). 
\]
Define 
\begin{equation}
\xi_N = \frac{1}{\sqrt{N}}\sum_{\ell=1}^Kc_\ell\Big[X(\ell)^*A_N(\ell)Y(\ell) - \rho(\ell)\mathrm{tr}(A_N(\ell))\Big] = \frac{1}{\sqrt{N}}\sum_{e}\psi_e. 
\end{equation}
Here $e = (u, v) \in \{1, 2, \ldots, N\}^2$ and 
\begin{equation}
\psi_{e} = \left\{
\begin{array}{ll}
\sum_{\ell=1}^Kc_\ell a^{(N)}_{uu}(\ell)[\overline{x}_{\ell u}y_{\ell u} - \rho(\ell)] & e = (u, u), \\
\sum_{\ell=1}^Kc_\ell a^{(N)}_{uv}(\ell)\overline{x}_{\ell u}y_{\ell v} & e = (u, v), u\neq v. 
\end{array}
\right.
\end{equation}
For any fixed $k\geq1$ we have
\begin{equation}\label{g}
N^{k/2}\xi_N^{k} = \sum_{e_1, \ldots, e_k}\psi_{e_1}\ldots\psi_{e_k} = \sum_{G}\prod_{e\in G}\psi_e := \sum_{G}\psi_G. 
\end{equation}
Here the directed graph $G=G(V, E)$ is defined by the vertex set $V = \{1, 2, \ldots, N\}$ and the edge set $E$ such that $(u\rightarrow v) \in E$ if and only if $e = (u, v)$ appeared in the summation. 

We shall use the method of moments to prove the theorem. That is, we will analyze the contribution of all $\mathbb{E}\psi_G$'s. 

For each graph $G$ in (\ref{g}), we can decompose it into several connected components. Just as in \cite{4}, these connected components can be classified into two types of sub-graphs. 
\begin{itemize}
\item
Type-I subgraph. 
\begin{definition}
If a connected component $C$ contains only one vertex, i.e., it only contains self-linked loops, then we call $C$ as a Type-I subgraph. We define
\begin{eqnarray*}
\mathcal{F}_1 & = & \text{ the set of all Type-I subgraphs in $G$. } \\
m_1 & = & |\mathcal{F}_1| = \text{ the number of Type-I subgraphs in $G$. } \\
\mu_1, \ldots, \mu_{m_1} & = & \text{ the degrees of vertices of these $m_1$ Type-I subgraphs. }
\end{eqnarray*}
\end{definition}
If $\mu_j = 2$ for some subgraphs of Type-I, then $\mathbb{E}\psi_G = 0$, hence it will not have any contribution to $\mathbb{E}\xi_N^k$. On the other side, if $\mu_j\geq4$ for all $j = 1, \ldots, m_1$, then 
\begin{equation}
\biggl|\mathbb{E}\prod_{C\in\mathcal{F}_1}\psi_C\biggl|\leq M\cdot (\epsilon_N N^{1/4})^{\sum_{i=1}^{m_1}(\mu_i-4)} 
\end{equation}
where $M$ is a sufficiently large constant.  
\item
Type-II subgraph. 
\begin{definition}
If a connected component $C_s$ contains at least one arrow $u\to v$ then we call it a Type-II subgraph. We define
\begin{eqnarray*}
\mathcal{F}_2 & = & \text{ the set of all Type-II subgraphs in $G$. } \\
m_2 & = & |\mathcal{F}_2| = \text{ the number of Type-II subgraphs in $G$. } \\
u_s & = & \text{ number of vertices for each subgraph $C_s\in\mathcal{F}_2, s = 1, 2, \ldots, m_2$ } \\
\gamma_{1s}, \ldots, \gamma_{u_ss} & = & \text{ the degrees of these $u_s$ vertices in $C_s$. }
\end{eqnarray*}
\end{definition}
If $\gamma_{js} = 1$ for some $j$ and some $s$, then we also have $\mathbb{E}\psi_G = 0$, giving no contribution to the overall expectation $\mathbb{E}\xi_N^{k}$. On the other side, if $\gamma_{js}\geq2$ for all $j, s$, then we have 
\begin{equation}
\biggl|\mathbb{E}\prod_{C_s\in\mathcal{F}_2}\psi_{C_s}\biggl| \leq M\cdot (\epsilon_N N^{1/4})^{\sum_{s=1}^{m_2}\sum_{j=1}^{u_s}(\gamma_{js}-2)}
\end{equation}
\end{itemize}
Now define $\mathcal{G}$ to be the set of graphs such that
\begin{eqnarray*}
\mathcal{G} & = & \mathcal{G}(m_1, \{\mu_j\}_{j=1}^{m_1}, m_2, \{u_s\}_{s=1}^{m_2}, \{\gamma_{js}\}_{1\leq j\leq u_s, 1\leq s\leq m_2}) \\
& = & \{G : \text{ $G$ has $m_1$ Type-I sub-graphs, with degree $\mu_j$ of each vertex.} \\
&& \text{ Also $G$ has $m_2$ Type-II subgraphs, with $u_s$ vertices in each Type-II sub-graph.} \\
&& \text{ Their degrees are defined by $\{\gamma_{js}\}_{1\leq j\leq u_s, 1\leq s\leq m_2}$}\}. 
\end{eqnarray*}
As the first observation, the number of all possibilities of different $\mathcal{G}$'s is a bounded constant, independent of $N$. From our previous analysis, for any $G\in\mathcal{G}$ we must have 
\begin{eqnarray}
|\mathbb{E}\psi_G| & \leq & M\cdot (\epsilon_NN^{1/4})^{\sum_{i=1}^{m_1}(\mu_i-4)+\sum_{s=1}^{m_2}\sum_{j=1}^{u_s}(\gamma_{js}-2)} \nonumber\\
& = & M\cdot (\epsilon_NN^{1/4})^{2k-4m_1-2\sum_{s=1}^{m_2}u_s}.
\end{eqnarray}
Here we used the equality $\sum_{i=1}^{m_1}\mu_i+\sum_{s=1}^{m_2}\sum_{j=1}^{u_s}\gamma_{js} = 2k. $

Next, we estimate the total number of graphs in $\mathcal{G}$, that is, we estimate $|\mathcal{G}|$. To get $G\in\mathcal{G}$, we need to pick $m_1$ vertices to form the Type-I subgraphs, having $O(N^{m_1})$ possibilities. Also we need to pick $m_2$ vertices to form Type-II subgraphs, having $O(N^{m_2})$ possibilities. Hence $|\mathcal{G}| = O(N^{m_1+m_2})$. 

Thus 
\begin{eqnarray}
N^{-k/2}\biggl|\sum_{G\in\mathcal{G}}\mathbb{E}\psi_G\biggl| &\leq& M\cdot (\epsilon_NN^{1/4})^{2k-4m_1-2\sum_{s=1}^{m_2}u_s}\cdot N^{m_1+m_2-k/2} \nonumber\\
& = & M\cdot\epsilon_N^{2k-4m_1-2\sum_{s=1}^{m_2}u_s}\cdot N^{-\sum_{s=1}^{m_2}(u_s-2)/2}.\label{contribution}
\end{eqnarray}
In order to have a non-negligible contribution on $\mathbb{E}\xi^{k}_N$, we need $\sum_{s=1}^{m_2}(u_s-2)/2\leq0$. However, we know $u_s\geq2$. Hence one must have $u_s = 2$ for all $s$. In this case we need $2k-4m_1-2\sum_{s=1}^{m_2}u_s\leq 0$ to guarantee a non-negligible contribution. From $\gamma_{js}\geq2$ and $\mu_i\geq4$ we obtain 
\[
2k = \sum_{i=1}^{m_1}\mu_i+\sum_{s=1}^{m_2}\sum_{j=1}^{u_s}\gamma_{js} \geq 4m_1+2\sum_{s=1}^{m_2}u_s. 
\]
Thus, in order for (\ref{contribution}) to be non-negligible, we necessarily need $u_s = 2$, $\mu_i = 4$ for all $1\leq i\leq m_1$ and $\gamma_{js}= 2$ for all $1\leq j\leq u_s, 1\leq s\leq m_2$. 

In summary, just as in \cite{4}, we proved that the non-negligible graphs will only consist of the following three connected components.
\begin{itemize}
\item
$k_1$ double loops $u\to u$ with terms $\mathbb{E}\Big[\sum_{\ell=1}^Kc_\ell a^{(N)}_{uu}(\ell)(\overline{x}_{\ell u}y_{\ell u}-\rho(\ell))\Big]^2$. 
\item
$k_2$ simple cycles $u\to v, v\to u$ with terms $\mathbb{E}\Big[\sum_{\ell=1}^Kc_\ell a^{(N)}_{uv}(\ell)\overline{x}_{\ell u}y_{\ell v}\Big]\Big[\sum_{\ell=1}^Kc_\ell \overline{a}^{(N)}_{uv}(\ell)\overline{x}_{\ell v}y_{\ell u}\Big].$
\item
$k_3$ double arrows $u\to v, u\to v$ with terms $\mathbb{E}\Big[\sum_{\ell=1}^Kc_\ell a^{(N)}_{uv}(\ell)\overline{x}_{\ell u}y_{\ell v}\Big]^2$.
\end{itemize}
We must have $4(k_1+k_2+k_3) = 2k$, or $k = 2(k_1+k_2+k_3)$ which must be an even number. Let $k = 2p$ for $p\in\mathbb{N}^+$. Similar to that in \cite{4} we have
\begin{equation}
\mathbb{E}\xi_N^{2p} = \frac{1}{N^p}\sum_{k_1+k_2+k_3=p}\frac{(2p)!}{k_1!k_2!k_3!}\biggl[\sum_{\{u_{j}^{(1)}\}_{j=1}^{k_1}}C_1\biggl]\biggl[\sum_{\{u_{j}^{(2)}, v_{j}^{(2)}\}_{j=1}^{k_2}}C_2\biggl]\biggl[\sum_{\{u_{j}^{(3)}, v_{j}^{(3)}\}_{j=1}^{k_3}}C_3\biggl] + o(1).
\end{equation}
Here
\begin{align*}
C_1 & =  \prod_{j=1}^{k_1}\mathbb{E}\biggl[\sum_{\ell=1}^Kc_\ell a^{(N)}_{u_j^{(1)}u_j^{(1)}}(\ell)(\overline{x}_{\ell u_j^{(1)}}y_{\ell u_j^{(1)}}-\rho(\ell))\biggl]^2, && \{u_j^{(1)}\}\subset \{1, \ldots, N\}.  \\
C_2 & =  \prod_{j=1}^{k_2}\mathbb{E}\biggl[\sum_{\ell=1}^Kc_\ell a^{(N)}_{u_j^{(2)}v_j^{(2)}}(\ell)\overline{x}_{\ell u_j^{(2)}}y_{\ell v_j^{(2)}}\biggl]\biggl[\sum_{\ell=1}^Kc_\ell \overline{a}^{(N)}_{u_j^{(2)}v_j^{(2)}}(\ell)\overline{x}_{\ell v_j^{(2)}}y_{\ell u_j^{(2)}}\biggl],
&&\{u_j^{(2)}, v_j^{(2)}\}\subset \{1, \ldots, N\}. \\
C_3 & =  \prod_{j=1}^{k_3}\mathbb{E}\biggl[\sum_{\ell=1}^Kc_\ell a^{(N)}_{u_j^{(3)}v_j^{(3)}}(\ell)\overline{x}_{\ell u_j^{(3)}}y_{\ell v_j^{(3)}}\biggl]^2, && \{u_j^{(3)}, v_j^{(3)}\}\subset \{1, \ldots, N\}.
\end{align*}
Hence we have 
\begin{eqnarray*}
\mathbb{E}\xi_N^{2p} & = & (2p-1)!!\sum_{k_1+k_2+k_3 = p}\frac{p!}{k_1!k_2!k_3!}D_1^{k_1}D_2^{k_2}D_3^{k_3} + o(1)\\
& = & (2p-1)!!(D_1+D_2+D_3)^p + o(1).
\end{eqnarray*}
Here 
\begin{eqnarray*}
D_1 & = & \sum_{\ell=1}^K\sum_{\ell'=1}^K c_\ell c_{\ell'}\mathbb{E}\Big[\overline{x}_{\ell 1}y_{\ell1}-\rho(\ell)\Big]\Big[\overline{x}_{\ell' 1}y_{\ell'1}-\rho(\ell')\Big]\cdot \frac{1}{N}\sum_{u=1}^Na^{(N)}_{uu}(\ell)a^{(N)}_{uu}(\ell') \\
& = & \sum_{\ell=1}^K\sum_{\ell'=1}^Kc_\ell c_{\ell'}\Big[\mathbb{E}\overline{x}_{\ell1}\overline{x}_{\ell'1}y_{\ell1}y_{\ell'1}-\rho(\ell)\rho(\ell')\Big]\omega_{\ell\ell'} + o(1). \\
D_2 & = & \sum_{\ell=1}^K\sum_{\ell'=1}^K c_\ell c_{\ell'}\mathbb{E}\Big[\overline{x}_{\ell1}y_{\ell'1}\Big]\mathbb{E}\Big[\overline{x}_{\ell'1}y_{\ell1}\Big]\cdot\frac{1}{N}\sum_{u\neq v}a^{(N)}_{uv}(\ell)\overline{a}^{(N)}_{uv}(\ell') \\
& = & \sum_{\ell=1}^K\sum_{\ell'=1}^K c_\ell c_{\ell'}\mathbb{E}\Big[\overline{x}_{\ell1}y_{\ell'1}\Big]\mathbb{E}\Big[\overline{x}_{\ell'1}y_{\ell1}\Big](\theta_{\ell\ell'}-\omega_{\ell\ell'}) + o(1). \\
D_3 & = & \sum_{\ell=1}^K\sum_{\ell'=1}^K c_\ell c_{\ell'}\mathbb{E}\Big[\overline{x}_{\ell1}\overline{x}_{\ell'1}\Big]\mathbb{E}\Big[y_{\ell1}y_{\ell'1}\Big]\cdot\frac{1}{N}\sum_{u\neq v}a^{(N)}_{uv}(\ell)a^{(N)}_{uv}(\ell') \\
& = & \sum_{\ell=1}^K\sum_{\ell'=1}^K c_\ell c_{\ell'}\mathbb{E}\Big[\overline{x}_{\ell1}\overline{x}_{\ell'1}\Big]\mathbb{E}\Big[y_{\ell1}y_{\ell'1}\Big](\tau_{\ell\ell'}-\omega_{\ell\ell'}) + o(1). \\
\end{eqnarray*}
This completes the proof of the theorem. 
\end{proof}

As a corollary, Theorem \ref{clt} can be applied to the situation where $A_N(\ell), x_i, y_i$ are all real matrices or vectors. In this case, we have $\theta_{\ell\ell'} = \tau_{\ell\ell'}$. The following corollary holds true. 

%
%
\begin{theorem}\label{real clt}
Under the setting of Theorem \ref{clt}, if in addition $A_N(\ell)\in\mathbb{R}^{N\times N}$ and $x_i, y_i\in\mathbb{R}^{K}$, then if the following limit exist
\begin{eqnarray}
\omega_{\ell\ell'} & = & \lim_{N\to\infty}\frac{1}{N}\sum_{u=1}^N a^{(N)}_{uu}(\ell)a^{(N)}_{uu}(\ell'), \\
\theta_{\ell\ell'} & = & \lim_{N\to\infty}\frac{1}{N}\sum_{u, v=1}^Na^{(N)}_{uv}(\ell){a}^{(N)}_{uv}(\ell')
\end{eqnarray}
Our vector $Z_N$ will converge weakly to Gaussian distributed random vector $W\in\mathbb{R}^K$ with the covariance matrix $B$  defined by $D = D_1+D_2$ where
\begin{eqnarray}
D_1 & = & \biggl(\Big[\mathbb{E}{x}_{\ell1}{x}_{\ell'1}y_{\ell1}y_{\ell'1}-\rho(\ell)\rho(\ell')\Big]\omega_{\ell\ell'}\biggl)_{\ell, \ell' = 1}^K, \\
D_2 & = & \biggl(\Big[\mathbb{E}[{x}_{\ell1}y_{\ell'1}]\mathbb{E}[{x}_{\ell'1}y_{\ell1}]+\mathbb{E}[{x}_{\ell1}{x}_{\ell'1}]\mathbb{E}[y_{\ell1}y_{\ell'1}](\theta_{\ell\ell'}-\omega_{\ell\ell'})\biggl)_{\ell, \ell'=1}^K. 
\end{eqnarray}
\end{theorem}

As a second simple corollary, we have the following result of convergence in probability.
%
%
\begin{corollary}\label{point}
In the setting of Theorem \ref{clt},  for any $0<\kappa<1/2$ we have 
\begin{equation}
N^{-\kappa}\Big[X(\ell)^*A_NY(\ell)-\rho(\ell)\mathrm{tr}(A_N(\ell))\Big] \to 0
\end{equation}
in probability. The similar result also holds for real cases. 
\end{corollary}

\section{Conclusion}\label{sec5}

In this paper, we studied the spiked population model to establish the asymptotic behavior of the sample eigenvalues and sample eigenvectors. The result is universal as we did not impose any strong assumptions on the original distribution of the sample points $x_i$. We showed that the joint distributions of the sample eigenvalues will be jointly normal, with the covariance matrix explicitly calculated. Also we showed that the entries of the sample eigenvectors will also be jointly normal. Finally. the angle between the sample eigenvector and the true eigenvector will converge to a non-trivial constant, with central-limit-theorem-style local fluctuation. All the covariance matrices for the limiting Gaussian distributions have been explicitly calculated. We showed that they only depend on the first four moments of the distribution of the sample points $x_i$. 

Under the special case where the sample points $x_i$ are Gaussian distributed, we showed as a corollary that the sample eigenvalues in different packs are asymptotically independent. Moreover the local fluctuation of the eigenvector is independent of the corresponding eigenvalue as well.

%
%


\begin{thebibliography}{1}
\bibitem{1} 
\textsc{J. Baik, G. Ben Arous, S. P\'{e}ch\'{e}}, 
\textit{Phase transition of the largest eigenvalue for nonnull complex sample covariance matrices},
Ann. Prob. 33 (5) 1643-97, 2005.

\bibitem{2} 
\textsc{S. P\'{e}ch\'{e}}, 
\textit{The largest eigenvalue of small rank perturbations of hermitian random matrices},
Prob. Theory Relat. Fields, 134 127-173, 2006.

\bibitem{3} 
\textsc{J. Baik, J. W. Silverstein}, 
\textit{Eigenvalues of large sample covariance matrices of spiked population models},
J. Multivariate Analysis 97 1382-1408, 2006.

\bibitem{4} 
\textsc{Z. D. Bai, J. F. Yao}, 
\textit{Central limit theorems for eigenvalues in a spiked population model},
Ann. de l'Institut Henri Poincar\'{e} -- Prob. et Stat. 44 (3) 447-474, 2008.

\bibitem{5} 
\textsc{D. Paul}, 
\textit{Asymptotics of sample eigenstructure for a large dimensional spiked covariance matrix model},
Statistica Sinica 17 1617-42, 2007.

\bibitem{6} 
\textsc{D. F\'{e}ral, S. P\'{e}che}, 
\textit{The largest eignvalue of rank one deformation of large Wigner matrices},
Comm. Math. Phys. 272 185-228, 2007. 

\bibitem{7} 
\textsc{M. Capitaine, C. Donati-Martin, D. F\'{e}ral}, 
\textit{The largest eigenvalues of finite rank deformation of large Wigner matrices: Covergence and nun-universality of the fluctuations},
Ann. Prob. 37 (1) 1-47, 2009. 

\bibitem{8} 
\textsc{F. Benaych-Georges, A. Guionnet, M. Maida}, 
\textit{Fluctuations of the extreme eigenvalues of finite rank deformations of random matrices},
Elec. J. Prob. 16 1621-62, 2011.


\bibitem{11} 
\textsc{P. Bianchi, M. Debbah, J. Najim}, 
\textit{Asymptotic independence in the spectrum of the gaussian unitary ensemble},
Elect. Comm. in Prob. 15 376-395, 2010.

\bibitem{12} 
\textsc{F. Bornermann}, 
\textit{Asymptotic independence of the extreme eigenvalues of gaussian unitary ensemble},
J. Math. Phys. 51 023514, 2010.

\bibitem{13} 
\textsc{V. A., Marc\v{e}nko, L. A. Pastur}, 
\textit{Distribution of eigenvalues for some sets of random matrices},
Math. USSR Sb. 1 457-486, 1967. 

\bibitem{14} 
\textsc{S. Geman}, 
\textit{A limit theorem for the norm of random matrices},
Ann. Prob. 8 (2) 252-261, 1980.

\bibitem{15} 
\textsc{J. W. Silverstein}, 
\textit{The smallest eigenvalue of a large dimensional wishart matrix},
Ann. Prob. 13 (4) 1364-1368, 1985.

\bibitem{16} 
\textsc{K. Johansson}, 
\textit{Shape ßuctuations and random matrices},
Comm. Math. Phys. 209 437Ð476, 2000. 

\bibitem{17} 
\textsc{Z. D. Bai, Y. Q. Yin}, 
\textit{Limit of the smallest eigenvalue of a large dimensional sample covariance matrices},
Ann. Prob. 21 1275-94, 1993.

\bibitem{18}
\textsc{C. Tracy, H. Widom},
\textit{Correlation functions, cluster functions and spacing distributions for random matrices},
J. Stat. Phys., 92 (5-6) 809-835, 1998.


\bibitem{19}
\textsc{Basor, Chen, Zhang},
\textit{PDEs satisfied by extreme eigenvalues distributions of GUE and LUE},
Random Matrices: Theory and Applications, 1 (1) 1150003, 2012. 


\bibitem{20}
\textsc{T. Baker, P. Forrester, P. Pearce},
\textit{Random matrix ensembles with an effective extensive external charge},
J. Phys. A 31 6087Ð6101, 1998.

\bibitem{21}
\textsc{I. Johnstone},
\textit{On the distribution of the largest eigenvalue in principal components analysis},
Ann. Stat. 29 (2) 295Ð327, 2001.



\bibitem{23} 
\textsc{Y. Q. Yin}, 
\textit{Limiting spectral distribution for a class of random matrices},
J. Multivariate Anal. 20 50-68, 1986.

\bibitem{24} 
\textsc{Z. D. Bai, Y. Q. Yin, P. R. Krishnaiah}, 
\textit{On the limit of the largest eigenvalue of the large dimensional sample covariance matrix},
Prob. Theory and Related Fields 78 509-521, 1988





\bibitem{a1}
\textsc{E. Telatar},
\textit{Capacity of multi-antenna gaussian channels}, 
European Trans. Telecommunications 10 (6) 585-595, 1999. 

\bibitem{a2}
\textsc{L. Lalous, P. Cizeau, M. Potters, J. Bouchaud},
\textit{Random matrix theory and financial correlations}, 
Internat. J. Theoret. Appl. Finance 3 (3) 391-397, 2000.

\bibitem{a3}
\textsc{A. Buja, T. Hastie, R. Tibshirani},
\textit{Penalized discriminant analysis}, 
Ann. Stat. 23 73-102, 2995. 

\bibitem{a4}
\textsc{R. Sear, J. Cuesta},
\textit{Instabilities in complex mixtures with a large number of components}, 
Phys. Rev. Lett. 91 (24) 245701, 2004.

\bibitem{a5}
\textsc{D. Hoyle, M. Rattray},
\textit{Limiting form of the sample covariance eigenspectrum in PCA and kernel PCA}, 
Advances in Neural Information Processing Systems NIPS 16, 2003.

\end{thebibliography}
\end{document}